\documentclass{amsart}

\title{Differential Borel equivariant cohomology via connections}
\author{Corbett Redden}
\address{Department of Mathematics, LIU Post, Long Island University, 720 Northern Blvd, Brookville, NY 11548, USA}
\email{corbett.redden@liu.edu}

\usepackage{hyperref} 
\hypersetup{ 
colorlinks,%
citecolor=black,%
filecolor=black,%
linkcolor=black,%
urlcolor=black } 

\usepackage{stmaryrd}  
\usepackage{amssymb}

\usepackage[all]{xy}  

\usepackage{comment}  

\usepackage{enumitem}

\usepackage[normalem]{ulem} 

\numberwithin{equation}{section}

\theoremstyle{plain}
\newtheorem{thm}[equation]{Theorem}
\newtheorem{prop}[equation]{Proposition}
\newtheorem{lemma}[equation]{Lemma}

\theoremstyle{definition}
\newtheorem{defn}[equation]{Definition}

\newtheorem{example}[equation]{Example}

\newtheorem{rem}[equation]{Remark}

\DeclareFontFamily{OT1}{pzc}{}
\DeclareFontShape{OT1}{pzc}{m}{it}%
              {<-> s * [1.15] pzcmi7t}{}
\DeclareMathAlphabet{\mathpzc}{OT1}{pzc}{m}{it}

\usepackage{mathtools}   

\newcommand{\Z}{\mathbb{Z}}
\newcommand{\R}{\mathbb{R}}

\newcommand{\fg}{\mathfrak{g}}
\newcommand{\iso}{\cong}
\renewcommand{\=}{:=}
\newcommand{\pt}{\text{pt}}
\newcommand{\into}{\hookrightarrow}
\newcommand{\cK}{\mathcal{K}}
\newcommand{\wt}{\widetilde}
\newcommand{\fk}{\mathfrak{k}}
\newcommand{\Id}{\operatorname{Id}}
\newcommand{\Map}{\operatorname{Map}}
\newcommand{\cC}{\mathcal{C}}
\newcommand{\Ker}{\operatorname{Ker}}
\newcommand{\cU}{\mathcal{U}}
\newcommand{\cH}{\mathcal{H}}
\newcommand{\cA}{\mathcal{A}}
\newcommand{\cB}{\mathcal{B}}
\newcommand{\Aut}{\operatorname{Aut}}
\newcommand{\End}{\operatorname{End}}
\renewcommand{\o}{\circ}

\newcommand{\Bn}{\mathpzc{B_{_\nabla}\!}}
\newcommand{\En}{\mathpzc{E_{_\nabla}\!}}

\newcommand{\B}{\mathpzc{B}}
\newcommand{\E}{\mathpzc{E}}

\newcommand\CG[1][G]{\check{C}_{#1}}
\newcommand\ZG[1][G]{\check{Z}_{#1}}
\newcommand\HG[1][G]{\check{H}_{#1}}

\newcommand\bH{\mathbb{H}}

\newcommand\x{\check{x}}

\renewcommand\L{\Lambda}
\newcommand\bas{\text{bas}}
\newcommand\basic{\operatorname{basic}}
\newcommand\hor{\text{hor}}

\newcommand\w{\operatorname{w}}

\newcommand\cc{I}
\newcommand\curv{R}
\newcommand\dR{\operatorname{dR}}
\renewcommand{\flat}{\operatorname{flat}}

\newcommand\D{\check}

\newcommand\Ad{\operatorname{Ad}}

\newcommand\dG{d_{G}}

\renewcommand\O{\Omega}
\newcommand\W{{\operatorname{W}}}

\newcommand\Obj{\operatorname{Ob}}
\newcommand\Der{\operatorname{Der}}

\newcommand\cF{\mathpzc{F}}
\newcommand\GMan{\text{-}\Man}

\newcommand\K{\operatorname{K}}
\newcommand\cM{\mathpzc{M}}

\newcommand*{\rchi}{\raisebox{0.35ex}{\( \chi \)}}

\newcommand\DGA{\operatorname{DGA}}

\newcommand\TV{T^V\hspace{-1mm}}
\renewcommand\TH{T^H\hspace{-1mm}}

\newcommand\cl{{\operatorname{cl}}}
\newcommand\CS{\operatorname{CS}}

\newcommand\Sing{\operatorname{Sing}_*\!}
\renewcommand\S{\mathcal{S}}
\DeclareMathOperator*{\hocolim}{hocolim}

\newcommand{\cat[1]}{\mathsf{#1}}

\newcommand\Man{\cat[Man]}
\newcommand\Top{\cat[Top]}
\newcommand\Set{\cat[Set]}
\newcommand\sSet{\cat[sSet]}
\newcommand\sPre{\cat[sPre]}

\newcommand\Gpd{{\cat[Gpd]}}
\newcommand\Gpdinfty{\infty\cat[Gpd]}
\newcommand\ho{\mathrm{ho}} 
\newcommand\DK{\Gamma} 
\newcommand\Stack{\cat[Stack]}
\newcommand\Shv{\cat[Shv]}  
\newcommand\PShv{\cat[PShv]}  

\newcommand\Bun[1][\T]{\cat[Bun]_{#1}}
\newcommand\Bunc[1][\T]{\cat[Bun]_{#1,\nabla}}
\newcommand\GBun[1][\T]{\text{-}\cat[Bun]_{#1}}
\newcommand\GBunc[1][\T]{\text{-}\cat[Bun]_{#1,\nabla}}

\newcommand\N{\operatorname{N}} 
\newcommand\h{\cat[h]} 
\newcommand\KZ[1][n]{\check{\mathcal{K}}(\Z,{#1})}
\renewcommand\K{\check{\mathcal{K}}}
\newcommand\op{\mathrm{op}}
\newcommand\Ch{\cat[Ch]}
\newcommand\Fun{\operatorname{Fun}}

\newcommand\Ab{\cat[Ab]}

\usepackage{bbm}
\renewcommand\1{\mathbbm{1}}

\newcommand*{\triple}[2][.1ex]{%
  \mathrel{\vcenter{\offinterlineskip%
  \hbox{$#2$}\vskip#1\hbox{$#2$}\vskip#1\hbox{$#2$}}}}

\newcommand*{\tripleleftarrow}{\triple{\leftarrow}}

\renewenvironment{itemize}{
  \begin{list}{$\bullet$}{
    \setlength{\leftmargin}{1.25em}
  }
}{
  \end{list}
}

\begin{document}


\begin{abstract}For a compact Lie group acting on a smooth manifold, we define the differential cohomology of a certain quotient stack involving principal bundles with connection.  This produces differential equivariant cohomology groups that map to the Cartan--Weil equivariant forms and to Borel's equivariant integral cohomology.  We show the Chern--Weil homomorphism for equivariant vector bundles with connection naturally factors through differential equivariant cohomology.
\end{abstract}

\maketitle
\thispagestyle{empty} 



\section{Introduction}

Differential cohomology, also known as Deligne cohomology \cite{Bry93} or Cheeger--Simons differential characters \cite{CS85}, is a contravariant functor $\HG[]^*$ from the category of manifolds to graded abelian groups.  It sits in the character diagram
\[ \xymatrix@R=3mm@C=6mm{ 0 \ar[dr]&&&&0 \\
&H^{n - 1}(M;\R/\Z) \ar[dr] \ar^{-B}[rr]&& H^{n}(M;\Z) \ar[ur]  \ar[dr]\\
H^{n -1}(M;\R) \ar[ur] \ar[dr]&& \HG[]^{n}(M) \ar[ur] \ar[dr] && H^{n}(M;\R) \\
& {\displaystyle \frac{\Omega^{n-1}(M)}{\Omega^{n-1}(M)_\Z}} \ar[ur] \ar_{d}[rr]&& \Omega^{n}(M)_\Z \ar[ur] \ar[dr]\\
0 \ar[ur]&&&&0
} \]
where the two diagonals are short exact sequences, $B$ is the Bockstein homomorphism, and the subscript $\Z$ denotes closed forms with $\Z$-periods.  Because it captures both the torsion of integral cohomology and the local structure of differential forms, differential cohomology has proven to be useful in a number of contexts.  In particular, the Chern--Weil homomorphism factors through $\HG[]^*$, making it a natural home for secondary invariants of principal bundles with connection.  Differential cohomology has now been generalized in several directions, leading to the notion of a ``differential extension'' of a generalized cohomology theory.

The purpose of this paper is to construct a differential extension of Borel's equivariant integral cohomology.  This was also recently done by K\"ubel \cite{1510.06392v1}, and our functors fit into the same short exact sequences.   Assume that $G$ is a compact Lie group acting smoothly on a manifold $M$.   The equivariant cohomology of $M$, with coefficients in an abelian group $A$, is the cohomology of the homotopy quotient $H_G^*(M;A) := H^*(EG\times_G M;A)$.  The complex of equivariant forms $(\O_G(M), d_G)$ is given by the Weil model $(S\fg^* \otimes \L \fg^* \otimes \O(M))^G_\hor$, whose cohomology is naturally isomorphic to $H^*_G(M;\R)$.  One can equivalently use the Cartan model $\left(S\fg^*\otimes \O(M) \right)^G$ in place of the Weil model, as the two complexes are isomorphic (not merely quasi-isomorphic).  We summarize the paper's main results in two theorems below.   Theorem \ref{thm:Properties} contains the important user-friendly properties of $\HG^*$.  It is simply Propositions \ref{prop:ses}, \ref{prop:GroupChange}, and \ref{prop:Reduction} combined into a single statement.  In fact, we show in Proposition \ref{prop:Uniqueness} that the functor $\HG^*$ is characterized uniquely by properties \ref{thm1diag} and \ref{thm1conn}.  Theorem \ref{thm:equivCCS} states that the equivariant Chern--Weil homomorphism, with equivariant forms defined by the equivariant curvature $\Omega_G$ \cite{MR705039}, factors through $\HG^*$.  For convenience, we restate it here as Theorem \ref{thm:ChernWeil} in the specific case of equivariant complex vector bundles.

 \begin{thm}\label{thm:Properties}For any compact Lie group $G$, there exists a contravariant functor $\HG^*(-)$, from the category of smooth $G$-manifolds to the category of graded abelian groups, satisfying the following.
 \begin{enumerate}[leftmargin=*, label=(\Alph*)]
 \item\label{thm1diag} There exist natural homomorphisms forming the following character diagram.
\[ \xymatrix@R=3mm@C=6mm{ 0 \ar[dr]&&&&0 \\
&H_G^{n - 1}(M;\R/\Z) \ar[dr] \ar^{-B}[rr]&& H_G^{n}(M;\Z) \ar[ur]  \ar[dr]\\
H_G^{n -1}(M;\R) \ar[ur] \ar[dr]&& \HG^{n}(M) \ar_{\cc}[ur] \ar^{\curv}[dr] \ar@{}_>>>>>{\dR}[rr]&& H_G^{n}(M;\R) \\
& {\displaystyle \frac{\Omega^{n-1}_G(M)}{\Omega^{n-1}_G(M)_\Z}} \ar[ur] \ar_{d_G}[rr]&& \Omega^{n}_G(M)_\Z \ar[ur] \ar[dr]\\
0 \ar[ur]&&&&0
} \]
 \item\label{thm1group} For $G_1 \xrightarrow{\phi} G_2$ a Lie group homomorphism, a smooth map $M \xrightarrow{F} N$ that is $\phi$-equivariant induces homomorphisms 
 \[ \HG[G_2]^*(N) \xrightarrow{F^*_\phi}  \HG[G_1]^*(M)\] compatible with the character diagrams.
 \item\label{thm1conn} A principal $G$-bundle with connection $(P,\Theta) \to M$ induces homomorphisms 
 \[ \HG^*(P) \xrightarrow{\Theta^*} \HG[]^*(M)\] 
 compatible with the character diagrams.
 \end{enumerate}
 \end{thm}

 \begin{thm}\label{thm:ChernWeil}Let $(V,\nabla) \to M$ be a $G$-equivariant Hermitian vector bundle with $G$-invariant connection.  There exist natural classes $\check{c}_k(\Theta_G) \in \HG^{2k}(M)$ satisfying
 \begin{enumerate}
 \item $\cc\left( \check{c}_k(\Theta_G) \right) = c_k \left( EG \times_G V \to EG \times_G M \right) \in H^{2k}_G(M;\Z)$,
 \item $\curv \left( \check{c}_k(\Theta_G) \right) = c_k\left( \Omega_G\right) \in \O^{2k}_G(M)$.
 \end{enumerate}
 \end{thm}


The paper is organized as follows.  Appendix \ref{sec:Background}, which appears at the end of the paper, contains a review of principal bundles with connection and the equivariant de Rham complex $\O_G(M)$.  Readers unfamiliar with the Weil model may find it helpful to begin here.  While this appendix is expository and should probably be replaced by a citation to \cite{GuilleminSternberg99}, we use it to introduce notation and to emphasize the relationship between $\W(\fg)$ and differential forms on principal bundles with connection.

In Section \ref{sec:HS} we define $\HG^*(M)$ using the Hopkins--Singer \cite{HS05} cochain model for $\HG[]^*(EG\times_G M)$, but with $\O(EG\times_G M)$ replaced by $\O_G(M)$.  The short exact sequences and character diagram follow immediately.  The advantage of this construction is that one need not leave the world of cochain complexes and topological spaces.  When giving constructions such as the Chern--Weil refinement, however, one is forced to play the tedious game of choosing classifying maps and checking everything is well-defined.

To emphasize the geometric nature of our constructions, we begin working with the differential Borel quotient stack $\En G \times_G M$ in Section \ref{sec:QuotientStack}.  This is defined as a contravariant functor that associates to any test manifold $X$ the groupoid of principal $G$-bundles with connection $(P,\Theta)\to X$, together with an equivariant map $P\to M$.   In \cite{MR3049871}, Freed--Hopkins show that a natural map
\[ \O^n_G(M) \xrightarrow{\iso} \O^n(\En G \times_G M)\] 
is an isomorphism, and we proceed to give the second definition
\[ \HG^n(M) \= \HG[]^n(\En G \times_G M).\]  
The details of this definition involve simplicial sheaves and are contained in Section \ref{Section:DefnProof}.  The virtue is that maps between various stacks induce homomorphisms in differential cohomology, as we demonstrate in Section \ref{sec:Examples}.

For example, the associated bundle construction defines a natural map
\[ \En G_1 \times_{G_1} M \xrightarrow{F_\phi} \En G_2 \times_{G_2} N,\]
which induces the homomorphism \ref{thm1group} from Theorem \ref{thm:Properties}.  Similarly, if $(P,\Theta) \to M$ is a principal $G$-bundle with connection, the pullback gives a natural map 
\[ M \xrightarrow{\Theta} \En G \times_G P,\]
 and this induces the homomorphism in \ref{thm1conn}.  For $(Q,\Theta) \to M$ a $G$-equivariant principal $K$-bundle with $G$-invariant connection, the refined Chern--Weil homomorphism is induced by a natural map 
 \[ \En G \times_G M \longrightarrow \Bn K,\]
 as described in Section \ref{sec:EquivCW}.  Given a $G$-bundle with connection mapping to $M$, we pull back $(Q,\Theta)$ to form a $G\times K$-bundle with connection, and then we quotient by $G$.  The equivariant extensions of the Chern forms appear when forming the connection on the pullback of $(Q,\Theta)$.  We conclude Section \ref{sec:EquivCW} with a discussion of $\HG^n(M)$ as differential characters on $\En G \times_G M$, and we show $\HG^2(M)$ is naturally equivalent to isomorphism classes of $G$-equivariant principal $\R/\Z$-bundles on $M$ with invariant connection.
%

A  definition of differential equivariant cohomology was previously given by Gomi \cite{MR2147734}, but it did not fully incorporate $\O_G(M)$, leading to groups that are not isomorphic to the ones in this paper.  See \cite{1510.06392v1} for a more detailed discussion.  In fact, many of our results can also be found in K\"ubel's paper, including the short exact sequences and the refined equivariant Chern--Weil homomorphism.  However, the constructions and methods are different.  The two definitions in this paper involve the topological space $EG\times_G M$ and the stack $\En G \times_G M$, and most homomorphisms are induced by bundle constructions.  On the other hand, K\"ubel uses the simplicial manifold $\{ G^\bullet \times M\}$ as the model for the homotopy quotient.  We believe that having multiple viewpoints will prove to be useful, just as it has for ordinary differential cohomology. 

\section{Cochain model}\label{sec:HS}
Our first definition of $\HG^*(-)$ is given by using the Hopkins--Singer construction for $\HG[]^*(EG\times_G M)$, with $\O(EG\times_G M)$ replaced by the Weil model of equivariant forms $\O_G(M)$.  The definition, short exact sequences, and ring structure all follow immediately from the same arguments given in Sections 2 and 3 of the Hopkins--Singer paper \cite{HS05}, so we will keep the proofs brief.  (In fact, \cite{HS05} also contains our construction for the case where $M=\pt$.)  Readers unfamiliar with the  Weil model of equivariant forms $\O_G(M)$ can find a full exposition in Appendix \ref{sec:Background}.  As noted in Section \ref{subsec:CartanMQ}, the Weil model may be replaced, via the Mathai--Quillen isomorphism, by the isomorphic Cartan model $\left( S\fg^* \otimes \O(M)\right)^G$  if desired.  


%

As discussed in Section \ref{subsec:deRhamModel}, let $(EG, \Theta_{EG})  \to BG$ be a universal principal $G$-bundle with connection, given as a direct limit of smooth finite-dimensional bundles.  If $M$ is a $G$-manifold, the Weil homomorphism induces an inclusion 
\[ \O^*_G(M) \xhookrightarrow{\w(\Theta_{EG})\otimes 1} \O^*(EG\times_G M) \into C^*(EG\times_G M;\R)\]
that is a quasi-isomorphism of cochain complexes.  

For the moment, we will use general coefficients.  Let $\L \subset V$ be a completely disconnected subgroup of a (possibly graded) vector space, and denote $\O^*(-;V) = \O^*(-)\otimes V$.  

\begin{defn}The differential equivariant cochain complex $(\CG(q)^*(M;\L),d)$ is a homotopy pullback in the diagram
\begin{equation}\label{eq:htpysquare} \vcenter{ \xymatrix{\CG(q)^*(M;\L) \ar[d]^{\curv} \ar[r]^-{\cc} &  C^*(EG\times_G M;\L) \ar@{_(->}[d] \\ \O_G^{*\geq q}(M;V) \ar@{^(->}[r] &C^*(EG \times_G M;V). }}\end{equation}
Explicitly, this is defined so that for $k \geq q$,
\begin{gather*}
\CG(q)^k(M;\L) \= C^k(EG \times_G M;\L) \times C^{k-1}(EG\times_G M;V) \times \O_G^k(M;V) \\
d(c,h,\omega) \= (\delta c, \, \omega - c - \delta h, \, \dG \omega),
\end{gather*}
and for $k < q$ we restrict to the subcomplex with $\omega=0$.  

For $\x = (c,h,\omega)$, we call $c$ the \textit{characteristic cocycle} and $\omega$ the \textit{curvature}, and we say $\x$ is \textit{flat} if the curvature is zero.  The degree $n$ differential equivariant cohomology is defined as the abelian group 
\[ \HG^n(M;\L) \=  H^n\left(\CG(n)^*(M;\L) \right) = \frac{\ZG^n(M;\L)}{d\CG^{n-1}(M;\L)_{\flat}} .\]
\end{defn}

%

As one would expect, the groups $\HG^*(M;\L)$ lie in short exact sequences that are completely analogous to those for ordinary differential cohomology.  We denote the image of $H^n_G(M;\L) \to H^n_G(M;V)$ by $H_G^n(M;V)_\L$, and we denote the subgroup of closed $n$-forms with $\L$-periods by $\O_G^n(M;V)_\L$.  This is defined so that under the de Rham isomorphism, 
\[ \frac{\O^n_G(M;V)_\L}{\dG \O^{n-1}_G(M;V)} \xrightarrow{\iso} H^n_G(M;V)_\L. \]

We will primarily be concerned with $\L=\Z$, $V=\R$, and we use the notations 
\[  \HG^*(M) = \HG^*(M;\Z), \quad  \HG^* = \HG^*(\pt;\Z).\]
In later sections, we will simply use $\Z \subset \R$ to simplify notation, but all results will immediately generalize to general coefficients $\L \subset V$.

\begin{prop}\label{prop:ses}The groups $\HG^*(M;\L)$ lie in the short exact sequences
\begin{align}
\label{ses1}\tag{SES 1}  0 \longrightarrow H^{n- 1}_G(M; V/\L) \longrightarrow &\HG^n(M;\L) \overset{\curv}\longrightarrow \O^n_G(M;V)_\L \longrightarrow 0, \\
\label{ses2}\tag{SES 2} 0 \longrightarrow \frac{\O^{n -1}_G(M;V)}{\O^{n -1}_G(M;V)_\L} \longrightarrow &\HG^n(M;\L) \overset{\cc}\longrightarrow H^n_G(M;\L) \longrightarrow 0, \\
\label{ses3}\tag{SES 3} 0 \longrightarrow \frac{H^{n - 1}_G(M;V)}{H^{n - 1}_G(M;V)_\L} \longrightarrow &\HG^n(M;\L) \longrightarrow A^n_G(M;V) \longrightarrow 0,
\end{align}
where $A^n_G(M)$ is the pullback (in sets) of the commutative square
\begin{equation}\label{eq:FundSquare} \vcenter{ \xymatrix{ \HG^n(M;\L) \ar@{->>}[r]^-{\cc} \ar@{->>}[d]^{\curv}& H^n_G(M;\L)  \ar@{->>}[d] \\
\O^n_G(M)_\L \ar@{->>}[r]^-{\dR} & H^n_G(M;V)_\L.
} } \end{equation}
These sequences fit into the character diagram
\[ \xymatrix@R=3mm@C=6mm{ 0 \ar[dr]&&&&0 \\
&H_G^{n - 1}(M;V/\L) \ar[dr] \ar^-{-B}[rr]&& H_G^{n}(M;\L) \ar[ur]  \ar[dr]\\
H_G^{n -1}(M;V) \ar[ur] \ar[dr]&& \HG^{n}(M;\L) \ar_-{\cc}[ur] \ar^-{\curv}[dr] && H_G^{n}(M;V) \\
& {\displaystyle \frac{\Omega^{n-1}_G(M;V)}{\Omega^{n-1}_G(M;V)_\L}} \ar[ur] \ar_-{d_G}[rr]&& \Omega^{n}_G(M;V)_\L \ar[ur] \ar[dr]\\
0 \ar[ur]&&&&0
} \]
where $B$ is the Bockstein homomorphism.
\begin{proof}The proof is straightforward and follows from the exact same arguments as those given in \cite{HS05}.  Alternatively, one can use the homological framework of Harvey--Lawson spark complexes, specifically Proposition 1.3 from \cite{HL06}.  The key point is that $\O_G^*(M) \to C^*(EG\times_G M;\R)$ is an inclusion, as noted in Lemma \ref{lem:WeilEGinj}, and it induces an isomorphism in cohomology.
\end{proof}\end{prop}

\begin{prop} \label{prop:HS-Properties} \
\begin{enumerate}[leftmargin=*]
\item\label{item1:HS-Prop} The constructions $\CG^*(M;\L)$ and $\HG^*(M;\L)$ define contravariant functors from $G\GMan$ to cochain complexes and graded abelian groups, respectively.  
\item\label{item2:HS-Prop} An inclusion $\L_1 \into \L_2 \into V$ of totally disconnected subgroups of $V$ induces natural transformations
\[ \HG^*(-;\L_1) \longrightarrow \HG^*(-;\L_2). \]
\item\label{item3:HS-Prop} The cup product and wedge product induce algebra homomorphisms
\[ \HG^{k_1}(M;\L_1) \otimes \HG^{k_2}(M;\L_2) \longrightarrow \HG^{k_1+k_2}(M;\L_1\otimes \L_2),\]
where $\L_1 \otimes \L_2 \subset V_1 \otimes V_2$.
\item \label{item4:HS-Prop} The map $M \overset{\pi}\to \pt$ makes $\HG^*(M;\L)$ into a $\HG^*$-module.
\item\label{item5:HS-Prop} If $V$ is a ring with sub-ring $\L$, then $\HG^*(M;\L)$ is a graded commutative ring, and $\cc$ and $\curv$ are ring homomorphisms.
\end{enumerate}
\end{prop}
\begin{proof}
The constructions $\CG^*(M)$ and $\HG^*(M)$ are clearly functorial with respect equivariant maps $f\colon M\to N$, as demonstrated in Appendix \ref{sec:Background}, thus proving \eqref{item1:HS-Prop}.

An inclusion $\L_1 \into \L_2$ induces natural cochain maps $C^*(-;\L_1) \into C^*(-;\L_2)$, and this gives $\CG^*(-;\L_1) \into \CG^*(-;\L_2)$, thus proving \eqref{item2:HS-Prop}.

For \eqref{item3:HS-Prop}, first note that the homomorphism $\O_G^*(M) \to \O^*(EG\times_G M)$ is a homomorphism of DGAs, as the basic subcomplex of a $G^\star$-algebra is a sub-DGA.  Hence, the product structure from \cite{HS05} immediately defines the product
\begin{align*}
\CG(q_1)^{k_1}(M;\L_1) \otimes \CG(q_2)^{k_2}(M;\L_2) &\overset{\cdot}\longrightarrow \CG(q_1+q_2)^{k_1 + k_2}(M;\L_1\otimes \L_2) \\
\HG^{k_1}(M;\L_1) \otimes \HG^{k_2}(M;\L_2) &\overset{\cdot}\longrightarrow \HG^{k_1 + k_2}(M;\L_1\otimes \L_2)
\end{align*} by 
\[ (c_1, h_1, \omega_1) \cdot (c_2, h_2, \omega_2) \= (c_1 \cup c_2, \> (-1)^{|c_1|} c_1 \cup h_2 + h_1 \cup \omega_2 + B(\omega_1, \omega_2), \> \omega_1\wedge \omega_2). \]
Here $B(\omega_1, \omega_2) \in C^{k_1 + k_2 - 1}(EG\times_G M;\R)$ is any natural cochain homotopy between $\wedge$ and $\cup$, and we denote the product in $\O_G(M)$ by $\wedge$ in a slight abuse of notation.

For \eqref{item4:HS-Prop}, the map $M \overset{\pi}\to \pt$ is $G$-equivariant, so it induces natural ring homomorphisms $\HG^*(\pt;\Z) \overset{\pi^*_G}\to \HG^*(M;\Z)$.  The composition 
\[ \HG^* \otimes \HG^*(M;\L) \overset{\pi^*_G \otimes 1}\longrightarrow \HG^*(M;\Z) \otimes \HG^*(M;\L) \overset{\cdot}\longrightarrow \HG^*(M;\L)\]
 makes $\HG^*(M;\L)$ into an $\HG^*$-module.

Item \eqref{item5:HS-Prop} follows similarly. If $\L$ is a ring, the product in $\HG^*$ can be composed with the product in $\L$, giving 
\[ \HG^*(M;\L) \otimes \HG^*(M;\L) \overset{\cdot}\longrightarrow \HG^*(M; \L\otimes \L) \longrightarrow \HG^*(M;\L).\]
The homomorphisms $\cc$ and $\curv$ preserve products at the cochain level, so the induced maps $\HG^*(M;\L) \overset{\cc}\to H^*_G(M;\L)$ and $\HG^*(M;\L) \overset{\curv}\to \O^*_G(M;V)_\L$ are also ring homomorphisms.
\end{proof}



Let us now use Propositions \ref{prop:ses} and \ref{prop:HS-Properties} to analyze $\HG^*(M)$ in a few simple examples.  

\begin{example}If $G=e$ is the trivial group, then $\O_G(M) \iso \O(M)$, and we can take $EG= \pt$.  This recovers the usual differential cohomology
\[ \HG[e]^*(M;\L) \iso \HG[]^*(M;\L).\]
\end{example}

\begin{example}The point $\pt$ is trivially a $G$-manifold, and the Borel construction is $EG \times_G \pt \iso BG$.  For $G$ compact we have $H^{2k}(BG;\R) \iso \O^{2k}_G(\pt) = S^k(\fg^*)^G$ and $H^{2k+1}(BG;\R) = \O^{2k+1}_G(\pt) = 0$.  The short exact sequences \eqref{ses2} and \eqref{ses1} give isomorphisms
\begin{equation}\label{eq:HGpt} \HG^{2k} \iso H^{2k}(BG;\Z), \quad \quad
\HG^{2k+1} \iso H^{2k}(BG;\R/\Z).\end{equation}
Furthermore, \eqref{ses2} becomes
\begin{align}\label{ses2pt}
0 \longrightarrow \frac{S^k(\fg^*)^G}{S^k(\fg^*)^G_\Z} \longrightarrow &\HG^{2k+1} \longrightarrow H^{2k+1}(BG;\Z) \longrightarrow 0,
\end{align}
where $\frac{S^k(\fg^*)^G}{S^k(\fg^*)^G_\Z}$ can also be written as $H^{2k}(BG;\R)\otimes \R/\Z$.
\end{example}

\begin{example}
If $G=\Gamma$ is a finite group, the Lie algebra is trivial.  Hence, the equivariant forms $\O_\Gamma(M)$ are simply the invariant forms  $\O(M)^\Gamma$.  When $M$ is a point, \eqref{ses2} implies  
\[\HG[\Gamma]^* \iso H^*(B\Gamma;\Z). \]
\end{example}

\begin{example}
For $G = S^1$, we know that $H^*(BS^1;\Z) \iso H^*(BS^1;\R)_{\Z} \iso \Z[t]$, where $|t|=2$.  The above identities \eqref{eq:HGpt} and \eqref{ses2pt} imply
\begin{align*} \HG[S^1]^{2k} &\iso H^{2k}(BS^1;\Z) \iso \Z t^k, \\
\HG[S^1]^{2k+1} &\iso H^{2k}(BS^1;\R)\otimes \R/\Z \iso (\R/\Z) t^k.
\end{align*}
We now show that as a graded commutative ring, this can be written
\[ \HG[S^1]^* \iso \Big( \Z \oplus (\R/\Z)\theta  \Big)[t], \]
where $|\theta|=1$ and $|t|=2$.  The relations $\theta t = t \theta$ and $\theta^2 = 0$ are implied by the adjectives graded commutative.

To verify this is a ring isomorphism, first note that the characteristic class map gives a ring isomorphism $\HG[S^1]^{2*} \iso \Z[t]$ in even degrees.  We can use the product's definition to determine what happens when multiplying by odd-degree elements.  Choose some cocycle $\alpha \in Z^2(BS^1;\Z)$ representing the generator of $H^2(BS^1;\Z)$.  An even element $n t^k \in \Z t^k \iso \HG[S^1]^{2k}$ can be represented by a cocycle of the form $(\alpha^k, b, t^k)$, and an odd element $[r] \theta t^k \in (\R/\Z) \theta t^k \iso \HG[S^1]^{2k+1}$ can be represented by $(0, r \alpha^k, 0)$.  The cochain-level product yields
\begin{align*}
 (n\alpha^{k_1}, \> b,\> n t^{k_1}) \cdot (0, \> r \alpha^{k_2},\> 0) &= (0, \> nr \alpha^{k_1 + k_2},\> 0), \\
 (0, \> r_1 \alpha^{k_1}, \>0) \cdot (0, \>r_2 \alpha^{k_2}, \>0) &= (0,\>0,\>0).
\end{align*}
Thus, the product of two odd elements is 0, and $(n t^{k_1}) \cdot ([r] \theta t^{k_2}) = [nr] \theta t^{k_1+k_2}$.
\end{example}

\begin{example}
Similarly, the odd cohomology groups of $BU(n)$ vanish, and
\[H^{*}(BU(n);\Z) \iso H^*(BU(n);\R)_\Z \iso \left(S^{*/2} \mathfrak{u} (n)^*\right)^{U(n)}_\Z \iso \Z[c_1, \ldots, c_n]. \]
The same argument given in the previous example shows that 
\begin{align*} \HG[U(n)]^{2k} &\iso H^{2k}(BU(n);\Z), \\
\HG[U(n)]^{2k+1} &\iso H^{2k}(BU(n);\R))\otimes \R/\Z,
\end{align*}
and as graded commutative rings with $|c_i| = 2i$ and $|\theta|=1$, 
\[ \HG[U(n)]^* \iso \Big(\Z\oplus (\R/\Z) \theta \Big) [c_1, \ldots, c_n].\]
\end{example}





\section{The differential quotient stack}\label{sec:QuotientStack}
From the description of the Weil algebra $\W(\fg)$ in Appendix \ref{sec:Background}, it is clear that equivariant forms $\O_G(M)$ most naturally arise when considering principal bundles {\it with connection}.  While the topological space $EG\to BG$ can be viewed as a universal bundle with connection, it does not naturally represent the {\it category} of principal bundles with connection on general manifolds.  For this reason, it will be more convenient to replace $EG\times_G M$ with the stack $\En G \times_G M$ and study its differential cohomology.  While we are not aware of any work specifically on $\HG[]^*(\En G \times_G M)$,  the general idea of defining differential cohomology via sheaves on manifolds has been widely used. This includes, but is not limited to, \cite{1208.3961, 1311.3188, MR3335251, MR3019405}.

\subsection{Background}\label{subsec:Sheaves}

First, remember that a \textit{groupoid} is a category in which every morphism is invertible.  Any \textit{set} $S$ can be viewed as a category, where the only morphisms are the identity morphisms.  We denote this fully faithful embedding of the category of sets into the category of groupoids by $\Set \subset \Gpd$.

A \textit{stack} $\cM$ (on the site of smooth manifolds) is a contravariant functor from the category of manifolds to the category of groupoids satisfying a descent condition; we denote the category by
\[ \Stack = \Shv_\Gpd \subset \Fun(\Man^\op, \Gpd).\]
This means that associated to every manifold $X$ is a groupoid $\cM(X)$, and a smooth map $X_1 \xrightarrow{f} X_2$ induces a functor 
\[ \cM(X_2) \xrightarrow{f^*} \cM(X_1).\]
There are associative natural transformations $g^* f^* \iso (g\o f)^*$ to deal with composition of functions.  To be a stack, as opposed to a prestack, $\cM$ must also satisfy a sheaf/descent condition that all of our examples will satisfy (see \cite{MR2206877} or Section \ref{Section:DefnProof} for more details).

\begin{example}
Any smooth manifold $N$ defines a stack $\underline{N}$ by associating to $X$ the set 
\[ \underline{N}(X) = C^\infty(X, N) \in \Set,\]
and to a smooth map $X_1 \xrightarrow{F} X_2$ the pullback $C^{\infty}(X_2, N) \xrightarrow{F^*} C^\infty (X_1, N)$.
\end{example}

\begin{example}Principal $G$-bundles with connection form a stack, which we denote $\Bn G$.  To any manifold $X$, let $\Bn G(X)$ be the groupoid whose objects are principal $G$-bundles $P\xrightarrow{\pi}X$ with connection $\Theta \in \O^1(P;\fg)$.  A \textit{morphism} $(P_1,\Theta_1) \overset{\varphi}\to (P_2,\Theta_2)$ is a bundle map preserving the connection; i.e. it is a $G$-equivariant map $\varphi$
\[ \xymatrix{P_1 \ar[rr]^{\varphi} \ar[dr]_{\pi_1}&& P_2 \ar[dl]^{\pi_2} \\ &X}\]
such that $\varphi^*(\Theta_2) = \Theta_1$.  Such a $\varphi$ must be a diffeomorphism, and hence all morphisms in $\Bn G(X)$ are isomorphisms.  Since bundles and connections pull back, a smooth map $f\colon X_1 \to X_2$ induces a functor $f^*\colon  \Bn G(X_2) \to \Bn G(X_1)$.  The groupoid of principal $G$-bundles on $X$, without connection, is defined analogously and denoted $\B G$.
\end{example}

The collection of \textit{morphisms between stacks} $\Shv_{\Gpd}(\cM_1, \cM_2)$ naturally forms a groupoid.  An object $\cM_1 \xrightarrow{\Psi} \cM_2$ is a collection of functors $\cM_1(X) \xrightarrow{\Psi(X)} \cM_2(X)$ for all $X$, together with natural transformations $\Psi(f)\colon  \Psi(X)\o f^* \to f^* \o \Psi(Y)$ for all smooth $f\colon  X\to Y$; morphisms between morphisms are given by natural transformations, which must be invertible since $\cM_2(-)$ is always a groupoid.


The Yoneda Lemma states that there is a canonical equivalence of categories
\[ \cM(X) \approx \Shv_{\Gpd}(\underline{X}, \cM),\]
and this defines a faithful embedding of the category of manifolds into the 2-category of stacks.  For this reason, we will usually not distinguish between a manifold $X$ and its associated stack $\underline{X}$.  We can also view an object in the category $\cM(X)$ as a map $X \to \cM$.  Thus, maps between stacks are  a generalization of smooth maps between manifolds; when $M$ and $X$ are both manifolds,
\[ \Shv_{\Gpd}(\underline{X}, \underline{M}) \iso \underline{M}(X) = C^\infty(X, M). \]

Via Yoneda, any bundle $(P,\Theta) \to X$ is naturally equivalent to a map 
\[ X \longrightarrow \Bn G, \]
and the groupoid of bundles with connection is naturally equivalent to the ``category of maps" $X \to \Bn G$.  This makes $\Bn G$ a more convenient classifying object for many purposes than the ordinary topological space $BG$.

\begin{example}Differential forms of degree $k$ define a stack $\O^k$, where 
\[  \O^k(X) \in \Set\]
is a set viewed as a groupoid with only identity morphisms.  A differential form $\omega \in \O^k(X)$ is equivalent to a morphism $X \xrightarrow{\omega} \O^k$.
\end{example}

%
%
%

\begin{example}\label{ExDefn:EnGM}If $M$ is a (left) $G$-manifold, the connection quotient stack $\En G \times_G M$ is defined as follows.  An object of $\left( \En G\times_G M \right)(X)$ is a principal $G$-bundle with connection $(P, \Theta) \to X$, together with a $G$-equivariant map $f\colon  P \to M$.  The map $f$ is equivalent to a section $F$ of the associated fiber bundle $P\times_G M$, with the equivalence given by $F(x) = [p, f(p)] \in P\times_G M$.  
\begin{align*}
X \to \En G \times_G M \quad \longleftrightarrow \quad  \vcenter{\xymatrix{ (P,\Theta) \ar[r]^{f} \ar[d] & M \\ X
}} \quad \longleftrightarrow \quad \vcenter{\xymatrix{ (P,\Theta) \times_G M \ar[d] \\ X \ar@/_1pc/[u]_{F} }}
\end{align*}
A morphism 
\[ \left(X \leftarrow (P_1, \Theta_1) \xrightarrow{f_1} M \right) \xrightarrow{\,\varphi\,} \left(X \leftarrow (P_2, \Theta_2) \xrightarrow{f_2} M \right) \] is a connection-preserving bundle isomorphism that covers the maps to $M$; i.e.
\[ \xymatrix{ &(P_1, \Theta_{1}) \ar[dd]_{\varphi}^{\iso} \ar[dl] \ar[dr]^{f_1}\\
X &&M\\
&(P_2, \Theta_{2}) \ar[ul] \ar[ur]_{f_2}
} \]
with $\varphi^*(\Theta_2) = \Theta_1$.

In the case where $M=\pt$, there is only one map $P \xrightarrow{f}\pt$, leading to the natural isomorphism $\En G \times_G \pt \iso \Bn G$.  We also let $\E G \times_G M$ denote the stack defined analogously via bundles without connection.
\end{example}

\begin{rem}Other possible notations for this connection quotient stack could include $M\!  \sslash^\nabla \! G$ and $(M_G )_\nabla$.  We choose to use $\En G\times_G M$ to emphasize the correspondence between maps of stacks and maps of topological spaces.  
\end{rem}

\subsection{Differential forms on the quotient stack}\label{subsec:FreedHopkins}
We want to study the differential cohomology of $\En G \times_G M$, but it is not immediately clear what that means.  As a first step, let's understand differential forms.  Since $\O^n(X) \iso \Shv_{\Gpd}(X, \O^n)$, it is natural to define 
\begin{equation}\label{defn:dRStack} \O^n(\En G \times_G M) \= \Shv_{\Gpd}(\En G \times_G M, \O^n).\end{equation}
An element of $\O^n(\En G \times_G M)$ is equivalent to functorially assigning, for all manifolds $X$, an element of $\O^n(X)$ to every object in $(\En G \times_G M)(X)$.  Since $\O^n(X)$ has only identity morphisms, isomorphic objects in $(\En G \times_G M)(X)$ must give equal elements in $\O^n(X)$.

Based on the  description of $\O_G(M)$ in Appendix \ref{sec:Background}, there is an obvious homomorphism
\begin{equation}\label{eq:EnGForms} \O_G^n(M) \overset{\iso}\longrightarrow \O^n( \En G \times_G M). \end{equation}
While it is not difficult to show this is injective, Freed--Hopkins prove it is in fact an  isomorphism of abelian groups \cite[Theorem 7.28]{MR3049871}.  Though $\O_G(M)$ is usually regarded as an algebraic replacement for $\O(EG\times_G M)$, it is the actual de Rham complex of the stack $\En G \times_G M$.  

We now explain the homomorphism \eqref{eq:EnGForms}.  To a map $X \to \En G \times_G M$, which is an object of $(\En G \times_G M)(X)$, the composition $\Theta^* \o f^*_G = \Theta^* \otimes f^*$ defines a homomorphism $\O^n_G(M) \to \O^n(X)$, as indicated in the following diagram.  
\begin{align*}
\vcenter{ \xymatrix{(P,\Theta) \ar[r]^f \ar[d] & M \\ X} }  \quad &\mapsto \quad \vcenter{\xymatrix{ \O^n_G(P) \ar[d]^{\Theta^*} & \O^n_G(M) \ar[l]_{f_G^*} \\ \O^n(P)_{\basic} \iso \O^n(X)}  }\end{align*}
More explicitly, for $\omega \otimes \eta \in \W^{2i,j}(\fg)$ and $\psi \in \O^k(M)$, then 
\[ \omega \otimes \eta \otimes \psi \mapsto \omega(\Omega^{\wedge i}) \wedge \eta( \Theta^{\wedge j}) \wedge f^*\psi \in \pi^* \O^{2i+j+k}(X).\]
This construction is invariant under morphisms in $(\En G \times_G M)(X)$.  If 
\[ \xymatrix{ &(P_1, \Theta_{P_1}) \ar[dd]_{\phi}^{\iso} \ar[dl]_{\pi_1} \ar[dr]^{f_1}\\
X &&M\\
&(P_2, \Theta_{P_2}) \ar[ul]^{\pi_2} \ar[ur]_{f_2}
} \]
with $\phi^*(\Theta_2) = \Theta_1$, it immediately follows that 
\[ \phi^*\big( \Theta_2^* \otimes f_2^* \big) = \phi^*\Theta_2^* \otimes (f_2 \o \phi)^* = \Theta_1^* \otimes f_1^*.\]
Since $\phi$ covers the identity on $X$, the two homomorphisms $\O^n_G(M) \to \O^n(X)$ are equal.

The de Rham differential induces a universal map of stacks $d\colon  \O^n \to \O^{n+1}$.  This gives a differential $\O^n(\En G \times_G M) \xrightarrow{d} \O^{n+1}(\En G \times_G M)$, and $(\O^*(\En G \times_G M), d)$ is naturally isomorphic to $(\O^*_G(M), d_G)$ as a cochain complex.

\begin{rem}It is important to note that the cochain complex $\O(\En G \times_G M)$ is given by 
\[ \cdots \xrightarrow{d} \Shv_\Gpd (\En G \times_G M, \O^{n}) \xrightarrow{d}  \Shv_\Gpd (\En G \times_G M, \O^{n+1}) \xrightarrow{d} \cdots.\]
Each degree is computed individually, with $\O^n$ being set valued, as opposed to some version of maps maps to the complex $\O^*$, where there are non-trivial morphisms.
\end{rem}

\begin{example}In the case where $M=\pt$, we have that $(S \fg^*)^G \iso \O(\Bn G)$, which is concentrated in even degrees.  A bundle with connection $(P,\Theta) \to X$ is naturally viewed as a map $X \to \Bn G$;  the induced map on differential forms is the usual map from Chern--Weil theory
\[ \xymatrix@R=0pt{ (S^n \fg^*)^G \ar[r]^{\iso} & \O^{2n}(\Bn G) \ar[r]^-{(P,\Theta)^*} &\O^{2n}(X) \\
\omega \ar@{|->}[rr]&& \omega(\Omega^{\wedge n}). } \]
\end{example}

\subsection{Differential cohomology via sheaves}\label{subsec:sPreDefn}
We now briefly outline the construction of $\HG[]^n(\En G \times_G M)$ and state the main properties.  Because the details are not important to understanding the constructions in Sections \ref{sec:Examples} and \ref{sec:EquivCW}, we postpone them until Section \ref{Section:DefnProof}.

Amongst its many constructions, ordinary cohomology can be represented by homotopy classes of maps to a  space.  For $A$ an abelian group, there exists an Eilenberg--MacLane space $K(A,n) \in \Top$ with 
\[ H^n(X;A) \iso [X, K(A,n)] = \ho \Top(X, K(A,n)).\]
While differential cohomology cannot be represented by a fixed topological space, it can be represented by a sheaf of ``spaces'' on the site of manifolds \cite{MR3019405,1208.3961, MR3335251, 1311.3188}.  We work with simplicial sets for convenience, as they form a natural home for groupoids, topological spaces, and chain complexes.  Letting $\Gpdinfty$ denote the $(\infty, 1)$-category of simplicial sets spanned by Kan complexes,  define $\K(\Z,n) \in \Shv_{\Gpdinfty}$ to be a sheaf fitting into a homotopy-commutative diagram of the form 
\begin{equation}\label{eq:SheafSquare} \vcenter{\xymatrix{ \K(\Z, n) \ar[d] \ar[r] & \cK(\Z,n) \ar[d] \\
\N(\O^n_{\cl})   \ar[r] & \cK(\R,n). }}\end{equation}
Here, $\cK(A,n)$ is an object in  $\Shv_{\Gpdinfty}$ representing ordinary cohomology.  

We can consider the sheaf of groupoids $\En G \times_G M$ as a sheaf of $\infty$-groupoids via the nerve construction $\N$.  The cohomology groups and differential cohomology groups are defined by considering maps between sheaves in the homotopy category:
\begin{align*} H^n(\En G \times_G M;A) &\= \ho \Shv_{\Gpdinfty}(\N(\En G \times_G M), \cK(A,n)),\\
\HG[]^n(\En G \times_G M) &\= \ho \Shv_{\Gpdinfty}(\N(\En G \times_G M), \K(\Z,n)) .\end{align*}

%
%
%


The following theorem is based on \eqref{eq:EnGForms} and the work of Bunke--Nikolaus--V\"olkl on homotopy-invariant sheaves \cite{1311.3188}.  The proof is given at the end of Section \ref{Section:DefnProof}.

\begin{thm}\label{thm:TwoDefnsAgree}
The two definitions of $\HG^*(M)$ are naturally equivalent,
\[ \HG[]^n(\En G \times_G M) \iso H^n\left( \CG^*(n)(M;\Z) \right).\]
\end{thm}

This leads to the following commutative square.
\begin{equation}\label{eq:StackSquare} \vcenter{\xymatrix{ \HG[]^n(\En G \times_G M) \ar[d] \ar[r] & H^n(\En G \times_G M; \Z) \iso H^n_G(M;\Z) \ar[d] \\ 
\O^n_G(M)_\cl \iso \O^n_{\cl}(\En G \times_G M) \ar[r] & H^n(\En G \times_G M; \R) \iso H^n_G(M;\R) } }\end{equation}
The isomorphism in the bottom-left corner is given by the Freed--Hopkins isomorphism \eqref{eq:EnGForms}.  The two isomorphisms on the right are induced by the map $EG\times_G M \to \En G\times_G M$, which induces an isomorphism in cohomology (Propositon \ref{prop:StackCohIso}).  In the following sections, we will use these two properties to check that maps of stacks induce the desired maps at the level of differential forms and cohomology.

\begin{example}Since $\O^{2k-1}(\Bn G) = 0$, the short exact sequence \eqref{ses2} gives an isomorphism
\[ \HG[]^{2k} (\Bn G) \xrightarrow{\iso} H^{2k}(BG;\Z).\]
The Cheeger--Chern--Simons refinement of Chern--Weil theory \cite{CS85} can then be repackaged in the following way.  A principal $G$-bundle with connection $(P,\Theta) \to X$ is equivalent to a map 
\[ X \xrightarrow{(P,\Theta)} \Bn G,\]
and this induces a homomorphism
\[ \HG[]^{2k}(X) \xleftarrow{(P,\Theta)^*} \HG[]^{2k}(\Bn G) \iso H^{2k}(BG;\Z).\]
Hence, any universal characteristic class in $H^{2k}(BG;\Z)$ has a canonical differential refinement.
\end{example}

\section{Constructions}\label{sec:Examples}
We now explain how some important constructions, frequently described using classifying spaces or equivariant forms,are naturally given by explicit geometric constructions involving bundles with connection.

\subsection{Associated bundles}\label{subsec:Associated}
Let $\phi\colon  G_1 \to G_2$ be any Lie group homomorphism, with $\phi_*\colon  \fg_1\to \fg_2$ the associated Lie algebra homorphism.  The associated bundle construction, which functorially makes any $G_1$-bundle with connection $(P,\Theta)$ into a $G_2$-bundle with connection, induces morphisms
\[ \xymatrix{ \En G_1 \ar[d] \ar[r] &\En G_2 \ar[d]\\
\Bn G_1 \ar[r] & \Bn G_2.}\]
The most common example is when $\phi$ is an inclusion, which is usually referred to as extending the structure group.  For more details on the associated bundle construction, see Chapter II.6 of \cite{KN63}, or Section 1 of \cite{Fre95}.

If $(P,\Theta) \in \Bn G_1 (X)$, the associated $G_2$-bundle is defined 
\begin{equation}\label{eq:AssocBundle} P_\phi \= P \times_{\phi} G_2 = (P\times G_2)/G_1 = \left( P \times G_2 \right)\big/ \big((pg,f) \sim (p,\phi(g)f)\big).\end{equation}
To define the induced connection $\Theta_\phi$ on $P_\phi$, first note that the natural map $P \xrightarrow{\varphi} P_\phi$, given by $p \mapsto [(p,1)]$, is $\phi$-equivariant.  Hence, the horizontal subspaces of $TP$ are mapped equivariantly into $TP_\phi$, and the image extends uniquely to an equivariant horizontal distribution in $TP_\phi$.  The induced connection $\Theta_\phi$ can also be described as the unique connection on $P_\phi$  that is compatible with $\Theta$ in the sense that 
\begin{equation}\label{eq:AssocConn} \varphi^*(\Theta_\phi) = \phi_* (\Theta) \in \O^1(P;\fg_2). \end{equation}
Therefore, the associated bundle construction is compatible with the Weil homomorphism, giving the commutative diagram
\begin{equation}\label{eq:WeilGroupChange} \vcenter{ \xymatrix{ \W(\fg_1) \ar[r]^{\Theta^*} &\O(P) \\
\W(\fg_2) \ar[u]_{\phi^*} \ar[r]^{\Theta_\phi^*} & \O(P_\phi). \ar[u]_{\varphi^*}
} }\end{equation}


\begin{prop}\label{prop:GroupChange}Let $F\colon  M \to N$ be a $\phi$-equivariant map, where $\phi\colon  G_1 \to G_2$ is a Lie group homomorphism.
\begin{enumerate}
\item \label{GroupChange1} The associated bundle construction gives a natural morphism
\[ \En G_1 \times_{G_1} M \xrightarrow{F_\phi} \En G_2 \times_{G_2} N.\]

\item \label{GroupChange2}The homomorphism induced by $F_\phi$ on differential forms is naturally isomorphic to $\phi^* \otimes F^*$ in the Weil model.
\[ \xymatrix@R=0cm{ \O_{G_2}(M) \ar[r]^{\iso} & \O(\En G_2 \times_{G_2} N) \ar[r]^{F_\phi^*} & \O(\En G_1 \times_{G_1} M)  & \O_{G_1}(M) \ar[l]_{\iso}\\
\cap &&& \cap \\
\W(\fg_2) \otimes \O(N) \ar[rrr]^{\phi^* \otimes F^*} &&& \W(\fg_1) \otimes \O(M) } \]

\item \label{GroupChange3}The homomorphism induced by $F_\phi$ on ordinary cohomology is naturally isomorphic to the homomorphism induced by the maps of classifying spaces
\[ EG_1 \times_{G_1} M \xrightarrow{} (EG_1 \times_\phi G_2) \times_{G_2} N \xrightarrow{} EG_2 \times_{G_2} N.\]

\item \label{GroupChange4} The induced homomorphism 
\[ \HG[G_2]^*(N) \xrightarrow{F_\phi^*} \HG[G_1]^*(M) \]
recovers the expected homomorphisms in equivariant cohomology and equivariant forms.
\end{enumerate}\end{prop}

Before giving the proof of Proposition \ref{prop:GroupChange}, we wish to first introduce notation for three special cases.  In two of the cases, the general construction simplifies, so we also show what happens when the stack is evaluated on a general manifold $X$. 
\begin{example} Suppose that $M$ and $N$ are $G$-manifolds, and $M \xrightarrow{F} N$ is $G$-equivariant.  The morphism of stacks
\begin{align*}&\En G \times_G M &&\xrightarrow{F_G} &&\En G \times_G N \\ 
&\vcenter{\xymatrix@R=.6cm{ (P,\Theta) \ar[r]^{f} \ar[d] & M \\ X
}}  &&\longmapsto  &&\vcenter{\xymatrix@R=.6cm{ (P,\Theta) \ar[d] \ar[r]^{f} & M \ar[r]^{F} &N \\ X  }}
\end{align*}
induces the homomorphism of abelian groups
\[ \HG^*(N) \xrightarrow{F_G^*} \HG^*(M).\]

\end{example}

\begin{example} Let $M$ be a $G_2$-manifold.  A homomorphism $G_1 \xrightarrow{\phi} G_2$ naturally defines a $G_1$-action on $M$ via $g_1 \cdot M \= \phi(g_1)\cdot M$, and the identity map $M \xrightarrow{\1}M$ is $\phi$-equivariant.  The associated bundle construction then induces
\begin{align*} \En G_1 \times_{G_1} M &\xrightarrow{\1_\phi} \En G_2 \times_{G_2} M, \\
\HG[G_1]^*(M) &\xleftarrow{\1_\phi^*} \HG[G_2]^*(M).
\end{align*}
\end{example}

\begin{example} Suppose that $K \vartriangleleft G$ is a normal subgroup, and suppose the quotient map $M \xrightarrow{q} M/K$ is a smooth map between manifolds.  Then, $q$ is equivariant with respect to the quotient group homomorphism $G \xrightarrow{/K} G/K$.  Quotienting $M$ and $G$ by $K$ induces the stack morphism 
\begin{align*}
&\En G \times_G M &&\xrightarrow{q_{/K}} &&\En (G/K) \times_{G/K} M/K \\ 
&\vcenter{ \xymatrix@R=.7cm{ (P,\Theta) \ar[d] \ar[r]^{f} &M \\ X }}  &&\xmapsto{\quad\,\,} &&\vcenter{\xymatrix@R=.7cm{(P,\Theta) \ar@{..>}[d] \ar@{..>}[r]^{f} &M \ar@{..>}[d]^{q} \\ (P_{/K}, \Theta_{/K}) \ar[r]^{f_{/K}} \ar[d] &M/K \\ X }}
\end{align*}
(with dotted arrows indicating maps no longer being used) and a homomorphism 
\[ \HG[G/K]^*(M/K) \xrightarrow{q_{/K}^*} \HG^*(M). \]  
If $K$ acts freely, this induces the usual isomorphism
\begin{align}\label{eq:QuotientFree}H^*_{G/K}(M/K;-) \xrightarrow[q_{/K}^*]{\iso} H^*_G(M;-) \end{align}
in equivariant cohomology.
\end{example}

\begin{proof}[Proof of Proposition \ref{prop:GroupChange}]
Part \eqref{GroupChange1} is given by constructing, for every $X$ and natural with respect to maps $X\to Y$, the following functor.
\begin{align*}  &\left(\En G_1 \times_{G_1} M\right) (X) &&\xrightarrow{F_\phi(X)} &&\left(\En G_2 \times_{G_2} N\right) (X) \\
&\vcenter{\xymatrix{ (P,\Theta) \times_{G_1} M \ar[d] \\ X \ar@/_1pc/[u]_{s} }}  &&\xmapsto{\quad \quad}
&&\vcenter{\xymatrix{ (P_\phi, \Theta_\phi) \times_{G_2} N \ar[d] \\ X \ar@/_1pc/[u]_{\left(\overline{\varphi \times F}\right) \o s}  }}
 \end{align*}
To explain in a bit more detail, the associated bundle construction associates to $(P,\Theta)$ the bundle $(P_\phi, \Theta_\phi)$, along with a $\phi$-equivariant map $P \xrightarrow{\varphi} P_\phi$.  Since $M \xrightarrow{F} N$ is also $\phi$-equivariant, the map
\[ P \times M \xrightarrow{ \varphi \times F} P_\phi \times N   \]
descends to the quotient
\[ P \times_{G_1} M \xrightarrow{ \overline{\varphi \times F}} P_\phi \times_{G_2} N,  \]
and we precompose with original section $s$.  The associated bundle construction is functorial, so any morphism $\Phi$ in $\left( \En G_1 \times_{G_1} M\right)(X)$ will induce a morphism $F_\phi(\Phi)$ in $\left( \En G_1 \times_{G_1} M\right)(X)$, as evidenced by the following commutative diagram.
\[  \xymatrix@R=.5cm{ &(P,\Theta) \times_{G_1} M \ar[dl] \ar[dd]_{\iso}^{\Phi} \ar[r]_{\overline{\varphi \times F}}& (P_\phi, \Theta_\phi) \times_{G_2} N \ar[dd]_{\iso}^{\Phi_\phi} \\
X \ar@/^1pc/[ur]^{s} \ar@/_1pc/[dr]_{s'} \\
&(P',\Theta') \times_{G_1} M  \ar[ul] \ar[r]^{\overline{\varphi' \times F}} & (P'_\phi, \Theta'_\phi) \times_{G_2} N }\]
Hence, the associated bundle construction induces a morphism of stacks.

For part \eqref{GroupChange2}, we simply trace through the construction of $F_\phi$ and the Weil homomorphism.  Let $\omega \in \O (\En G_2 \times_{G_2} N)$, which may be written as a sum of homogeneous elements $\alpha \otimes \psi \in \left(\W(\fg_2) \otimes \O(N)\right)_{\basic}$.  Then, calculate $F_\phi^*\omega \in \O(\En G_1 \times_{G_1} M)$ by evaluating on a test manifold $X$.  By definition, 
\begin{align*} &\left( \En G_1 \times_{G_1} M\right)(X) & &\xrightarrow{F_\phi} &&\left(\En G_2 \times_{G_2} N \right)(X) &&\xrightarrow{\omega} &&\O(X) \\
&\vcenter{ \xymatrix{ (P,\Theta) \times_{G_1} M \ar[d] \\ X \ar@/_1pc/[u]_{s} } }  &&\longmapsto &&
 \vcenter{ \xymatrix{ (P_\phi, \Theta_\phi) \times_{G_2} N \ar[d] \\ X \ar@/_1pc/[u]_{(\overline{\varphi \times F}) \o s}  }} &&\longmapsto & s^*  \Big( (\overline{\varphi \times F})^*& \left( \alpha(\Theta_{\phi}) \otimes \psi \right) \Big)
\end{align*}
As noted in \eqref{eq:WeilGroupChange}, the associated bundle construction is compatible with the Weil homomorphism, and therefore
\[
(\varphi \times F)^* \big( \alpha(\Theta_\phi) \otimes \psi \big) = \alpha( \varphi^*\Theta_\phi) \otimes F^*\psi 
= \alpha( \phi_*\Theta) \otimes F^*\psi = (\phi^*\alpha) (\Theta) \otimes F^*\psi.
\]
Consequently, $F_\phi^* \omega$ gets mapped to $s^*\big( (\phi^*\alpha) (\Theta) \otimes F^*\psi\big) \in \O(X)$, and this is the same element that $\phi^*\alpha \otimes F^*\psi \in \O_{G_1}(M)$ maps to.  

To see part \eqref{GroupChange3}, choose a classifying map for the $G_2$-bundle $EG_1 \times_\phi G_2$, which gives the following diagram.  We do not need to assume the classifying map is connection preserving.
\[ \xymatrix{ EG_1 \times_{G_1} M \ar[r] \ar[d] \ar@/^1pc/[rr]^{\Psi} & EG_1 \times_\phi G_2 \ar[r] \ar[d]& EG_2 \ar[d] \\ BG_1 \ar[r]^{=} & BG_1 \ar[r] & BG_2 
}\]
For any $X\xrightarrow{f}BG_1$, there is an isomorphism of associated bundles
\[ f^* (\Psi^* EG_2) \iso (f^*EG_1)_\phi.\]
This implies the diagram
\[  \xymatrix{ EG_1 \times_{G_1} M \ar[r] \ar[d] & EG_2 \times_{G_2} M \ar[d] \\ \E G_1 \times_{G_1} M \ar[r] & \E G_2 \times_{G_2} N } \]
commutative up to isomorphism, or homotopy-commutative.  As shown in Proposition \ref{prop:StackCohIso}, the vertical maps induce isomorphisms in ordinary cohomology.  Therefore, the induced map on the cohomology of the stacky homotopy quotients
\[ H^*(\En G_2 \times_{G_2} N;A) \xrightarrow{F^*_\phi} H^*(\En G_1 \times_{G_1} M;A) \] 
is given by the usual map between the homotopy quotient spaces.

Finally, part \eqref{GroupChange4} follows immediately from parts \eqref{GroupChange2} and \eqref{GroupChange3}, and \eqref{eq:StackSquare}.
\end{proof}

\subsection{Pulling back equivariant bundles}\label{subsec:Reduction} As discussed in the previous section, if $G$ acts freely on a manifold $P$, there is a natural homomorphism $\HG[]^*(P/G) \xrightarrow{q_{/G}^*} \HG^*(P)$ inducing the standard isomorphism in cohomology.  We now describe the one-sided inverse to this map.  First we outline the  result, then we give the construction, and then we check the details.

To obtain a map in the opposite direction, one must choose a connection $\Theta$ on $P$.  Using this and denoting $M=P/G$, one naturally obtains a morphism of stacks
\[ M \overset{\Theta}\longrightarrow \En G\times_G P \]
by pulling back the bundle $(P,\Theta)$, as indicated below when the stack is evaluated on a manifold $X$.
\begin{align*}
\vcenter{ \xymatrix@C=.5cm{ X \ar[r]^{f} &M }} &\longmapsto \vcenter{ \xymatrix{ f^*(P,\Theta) \ar[d] \ar[r]^{\wt{f}} &(P,\Theta) \ar@{..>}[d] \\ X \ar@{..>}[r]^{f} & M } }
\end{align*}
Proposition \ref{prop:Reduction} will imply that the induced map 
\[ \HG^*(P) = \HG[]^*(\En G \times_G P) \xrightarrow{\Theta^*} \HG[]^*(M) \]
recovers the isomorphism inverse to $q_{/G}^*$ at the level of cohomology, and it recovers the Weil homomorphism 
\begin{align*}
\O_G(P) = \big( \W(\fg)\otimes \O(P) \big)_{\basic} & \xrightarrow{\Theta^*} \O(P)_{\basic} \iso \O(M) \\
\alpha \otimes \psi \quad \quad &\longmapsto \quad \alpha(\Theta) \otimes \psi
\end{align*} 
at the level of differential forms.

The above is actually a special case of what happens when we have an equivariant principal bundle with connection.  Suppose that $(Q,\Theta) \to M$ is a $G$-equivariant principal $K$-bundle with $G$-invariant connection.  This implies that $Q$ is a $(G\times K)$-manifold, $M$ is a $G$-manifold, and the map $Q\to M$ is equivariant.  (Note that in the following discussion, $G$ will always act on the base manifold $M$, and $K$ will be the structure group for a principal bundle.)  The above construction generalizes to give a stack morphism
\[ \En G \times_G M \xrightarrow{\Theta_G} \En (G\times K) \times_{G\times K} Q, \]
and this gives the expected classical maps in cohomology and equivariant forms.  In particular, we get the following maps between short exact sequences.
\[ \xymatrix{ 0 \ar[r]& \frac{\O^{*-1}_{G\times K}(Q)}{\O^{*-1}_{G\times K}(Q)_\Z} \ar[r] \ar@/_1pc/[d]_{\Theta_G^*}&\HG[G\times K]^*(Q) \ar@/_1pc/[d]_{\Theta_G^*}\ar[r] & H^*_{G\times K}(Q;\Z) \ar[r] \ar@{<->}[d]^{\iso} &0 \\
0 \ar[r]& \frac{\O^{*-1}_G(M)}{\O_G^{*-1}(M)_\Z} \ar[r] \ar@/_1pc/[u]_{q_{/K}^*}&\HG[G]^*(M) \ar[r] \ar@/_1pc/[u]_{q_{/K}^*} & H_G^*(M;\Z) \ar[r] &0 
}\]
\[ \xymatrix{ 0 \ar[r] & H_{G\times K}^{*-1}(Q;\R/\Z) \ar@{<->}[d]^{\iso} \ar[r] & \HG[G\times K]^*(Q) \ar@/_1pc/[d]_{\Theta_G^*} \ar[r] & \O^*_{G\times K}(Q)_\Z \ar@/_1pc/[d]_{\Theta_G^*} \ar[r] &0 \\
0 \ar[r] & H_G^{*-1}(M;\R/\Z) \ar[r] & \HG[G]^*(M) \ar@/_1pc/[u]_{q_{/K}^*} \ar[r] & \O_G^*(M)_\Z \ar@/_1pc/[u]_{q_{/K}^*} \ar[r] &0 
} \]

Let us now give the precise construction in slightly greater generality.  Assume that $K$ is a normal subgroup of $\wt{G}$ with quotient $\wt{G}/K \iso G$, and furthermore assume that we have fixed a splitting of the Lie algebra, i.e.
\begin{align}\label{eq:KGassumption} \begin{split} 1 \to K \into &\wt{G} \to G \to 1 \\ \wt{\fg} \iso &\fg \oplus \fk. \end{split} \end{align}
Such a decomposition of Lie algebras must exist since our groups are compact.  Of primary interest is when $\wt{G} = G\times K$, but we wish to also allow examples such as $SU(n) \into U(n) \to U(1)$ or $U(1) \into {\rm Spin}^c(n) \to SO(n)$.

%

\begin{defn}For $M\in G\GMan$, let $\wt{G}\GBunc[K](M)$ be the groupoid of $\wt{G}$-equivariant principal $K$-bundles on $M$ with invariant connection.  An object $(Q,\Theta) \in \wt{G}\GBunc[K](M)$ is a principal $K$-bundle with connection on $M$ such that:
\begin{itemize}
\item $Q \in \wt{G}\GMan$ and $Q \overset{\pi}\to M$  is equivariant,
\item $\Theta \in \big( \O^1(Q) \otimes \fk\big)^{\wt{G}}$.
\end{itemize}
For $(Q_i, \Theta_i) \in \wt{G}\GBunc[K](M)$, a morphism $(Q_1,\Theta_1) \overset{\phi}\to (Q_2,\Theta_2)$ is a $\wt{G}$-equivariant map $Q_1 \overset{\phi}\to Q_2$ which is an isomorphism of $K$-bundles with connection.
\end{defn}

\begin{rem}For $\wt{G} = G \times K$, the condition $\Theta \in \big(\O^1(Q) \otimes \fk\big)^{G\times K}$ can be rewritten as $k^*\Theta = \Ad_{k^{-1}}\Theta$ and $g^*\Theta = \Theta.$  Also, in this case we may refer to $\wt{G}$-equivariant $K$-bundles as $G$-equivariant $K$-bundles.  This is a standard convention, and we hope it does not cause any confusion.  
\end{rem}

Given an element $(Q,\Theta) \in \wt{G}\GBunc[K](M)$, we want to define a natural map $\Theta_G$
\[ \xymatrix{ \En \wt{G} \times_{\wt{G}} Q \ar[r]_-{q_{/K}} & \En G \times_G M \ar@/_1pc/[l]_-{\Theta_G} 
} \]
such that the composition $q_K \o \Theta_G$ is naturally isomorphic to the identity.  We first describe the construction, hopefully emphasizing the geometric nature and naturalness.  The details, which are relatively straightforward, are then checked in Proposition  \ref{prop:Reduction}.  While the induced maps in cohomology and equivariant forms that we recover are well-known, and the construction of $\Theta_G$ appears implicitly in \cite{MR1850463}, we do not know any references where this construction is done in general at the level of principal bundles with connection.

We now give a functor between the two groupoids produced when the stacks are evaluated on a test manifold $X$.  Denoting $\Theta=\Theta_Q$ for added clarity, the map is given by the following construction.
\begin{align*}  &(\En G \times_G M ) (X) &&\xrightarrow{\Theta_G(X)} &&(\En \wt{G} \times_{\wt{G}} Q ) (X) \\
&\vcenter{ \xymatrix{& (Q,\Theta_Q) \ar[d]^{\pi_M} \\ (P, \Theta_P) \ar[r]^{f} \ar[d]^{\pi_X} & M \\X }} &&\xmapsto{\quad \quad}
&& \vcenter{ \xymatrix{\Big(f^*Q, \, \pi_P^*\Theta_P \oplus (\wt{f}^*\Theta_Q - \iota_{\pi_P^* \Theta_P}\wt{f}^*\Theta_Q) \Big) \ar[r]^-{\wt f} \ar@{..>}[d]^-{\pi_P} \ar@/_3pc/[dd]_{\pi'_X} &(Q,\Theta_Q) \ar@{..>}[d]^{\pi_M} \\ (P, \Theta_P) \ar@{..>}[r]^{f}  \ar@{..>}[d] & M \\X} }
 \end{align*}
Our new $\wt{G}$-bundle is given by pulling back $Q$.  The connection on $f^*Q$ is given by the pulling back the connections on $P$ and $Q$ and subtracting a correction term.  This extra term is defined to be the image 
\begin{align*}
\O^1(\wt{f}^*Q) \otimes \fg \otimes \O^1(\wt{f}^*Q) \otimes \fk & \longrightarrow \O^1(\wt{f}^*Q) \otimes \fk \\
\pi^*_P \Theta_P \otimes \wt{f}^*\Theta_Q &\longmapsto \iota_{\pi_P^* \Theta_P}\wt{f}^*\Theta_Q
\end{align*}
under the natural contraction $\fg \otimes \O^1(\wt{f}^*Q) \xrightarrow{\iota} \O^0(\wt{f}^*Q)$.  While this term may initially seem obscure, it is necessary to ensure that we produce a connection on $f^*Q$, and it has a simple description in the Weil model.

Define the equivariant extension of the connection and curvature forms by 
\begin{align}
\label{eq:equivConn} \Theta_G &\= \Theta - \iota_{\theta_\fg}\Theta &&\in \big( \W(\fg)\otimes \O(Q) \big)^1\otimes \fk, \\
\label{eq:equivCurv} \Omega_G &\= d_G \Theta_G + \tfrac{1}{2} [\Theta_G \wedge \Theta_G] &&\in  \big( \W(\fg)\otimes \O(Q) \big)^2 \otimes \fk.
\end{align}
To clarify, $\iota_{\theta_\fg} \Theta \in \L^1 \fg^* \otimes \O^0(Q) \otimes \fk$ and it is evaluated on elements $\xi \in \fg$ by contracting the connection along the vector field in $Q$ generated by the $G$-action; i.e.
\[ \langle \iota_{\theta_\fg} \Theta, \xi\rangle \= \iota_{\xi} \Theta \in \O^0(Q;\fk).\]
The $\fk$ portion of the connection on $\wt{f}^*Q$ is simply the image of the equivariant connection $\Theta_G$ under the Weil homomorphism 
\begin{align*}
\W(\fg)\otimes \O(Q) \otimes \fk &\xrightarrow{(\pi^*_P\Theta_P)^*\otimes \wt{f}^* \otimes 1} \O(\wt{f}^*Q) \otimes \fk \\
\Theta_G = \Theta - \iota_{\theta_\fg} \Theta &\xmapsto{\quad\quad\quad} \wt{f}^*\Theta_Q - \iota_{\pi_P^* \Theta_P}\wt{f}^*\Theta_Q.
\end{align*}

Note that the additional term $\iota_{\pi_P^* \Theta_P}\wt{f}^*\Theta_Q$ uses the connection on the $G$-bundle $P$ to detect the $G$-action on $Q$.  This is key to understanding the equivariant Chern--Weil homomorphism in Section \ref{sec:EquivCW}.

\begin{rem}\label{Rem:CartanCurv}In the Cartan model for equivariant forms (see Section \ref{subsec:CartanMQ}), the equivariant connection and curvature simplify to
\begin{align*}
\Theta_G &\longleftrightarrow \Theta &&\in S^0 \fg^* \otimes \O^1(Q)\otimes \fk, \\
\Omega_G &\longleftrightarrow \Omega - \iota_{\Omega_\fg} \Theta &&\in \Big(S^0 \fg^* \otimes \O^1(Q) + S^2\fg^* \otimes \O^0(Q)\Big) \otimes \fk,
\end{align*}
with $\Omega_G$ becoming the more familiar equivariant curvature defined in \cite[\textsection 2]{MR705039}.
\end{rem}

We proceed to check that $\Theta_G$ and $\Omega_G$ are equivariant forms and the above construction satisfies the desired properties.

\begin{lemma}\label{lem:equivConn}For $(Q,\Theta) \in \wt{G}\GBunc[K](M)$, the forms $\Theta_G$ and $\Omega_G$ are $\wt{G}$-invariant and $\fg$-horizontal; i.e. they live in the subcomplex $\left(\W(\fg)\otimes \O(Q) \otimes \fk\right)^{\wt{G}}_{\fg \text{-}\hor}$.  When $\wt{G} = G\times K$, this can be written
\[ \Theta_G \in \O^1_G(Q;\fk), \quad \Omega_G \in \O^2(Q;\fk).\]
\end{lemma}
\begin{proof}
To see that $\Theta_G$ is $\fg$-horizontal, let $\xi \in \fg$.  Then
\[ \iota_\xi \Theta_G = \iota_\xi \Theta - \iota_{\theta_\fg(\xi)} \Theta = \iota_\xi \Theta - \iota_\xi \Theta = 0.\]
We now show both terms in $\Theta_G$ are $\wt{G}$-invariant.  First note that $\Theta \in \left( \O^1(Q) \otimes \fk\right)^{\wt{G}}.$  To check the term $\iota_{\theta_{\fg}}\Theta$, observe that $\theta_{\fg} \in \left( \L^1 \fg^* \otimes \fg \right)^G = \left( \L^1 \fg^* \otimes \fg \right)^{\wt{G}}$, and the contraction map 
\begin{align*} \fg \otimes \O^1(Q) &\overset{\iota}\longrightarrow \O^0(Q) \\
\xi \otimes \omega &\longmapsto  \iota_{\xi} \omega 
\end{align*}
is $\wt{G}$-equivariant.  Hence,
\[ \iota_{\theta_{\fg}}\Theta \in \left( \L^1 \fg^* \otimes \O^0(Q) \otimes \fk\right)^{\wt{G}},\]
so $\Theta_G$ is $\fg$-horizontal and $\wt{G}$-invariant.

From the Leibniz rule for $\iota_\xi$ and the definition of tensor product of representations, it immediately follows that $[\Theta_G \wedge \Theta_G]$ is also $\fg$-horizontal and $\wt{G}$-invariant.  The standard  identities \eqref{eq:gstar1} and \eqref{eq:gstar2} imply that $d_G \Theta_G$ is also $\wt{G}$-invariant and $\fg$-horizontal.  Consequently, $\Omega_G$ is invariant and horizontal.
\end{proof}


\begin{prop}\label{prop:Reduction}Assume $\wt{G}/K \iso G$ satisfying \eqref{eq:KGassumption}.  Let $(Q,\Theta) \in \wt{G}\GBunc[K](M)$.
\begin{enumerate} 
\item\label{Reduction1} The above construction defines a natural morphism $\Theta_G$
\[ \xymatrix{ \En \wt{G} \times_{\wt{G}} Q \ar[r]_-{q_{/K}} & \En G \times_G M \ar@/_1pc/[l]_-{\Theta_G} 
} \]
such that the composition $q_{/K} \o \Theta_G$ is naturally isomorphic to the identity.
\item\label{Reduction2} The induced map 
\[ \O\big( \En \wt{G} \times_{\wt{G}} Q  \big) \xrightarrow{\Theta_G^*} \O\big( \En G \times_G M \big) \]
is equivalent to the homomorphism $\Theta_G^*\colon  \O_{\wt{G}}(Q) \to \O_G(M)$ induced from
\begin{align*}
\W(\fg)\otimes \W(\fk) \otimes \O(Q) &\overset{\Theta_G^*}\longrightarrow \W(\fg) \otimes \O(Q) \\
\alpha \cdot \beta \cdot \gamma &\longmapsto \alpha \cdot \beta(\Theta_G) \cdot \gamma.
\end{align*}

\item \label{Reduction3} The induced map in cohomology is the standard isomorphism
\[ H^*_{\wt{G}} (Q;-) \overset{\iso}\longrightarrow H^*_{G}(M;-). \]
\end{enumerate}
\end{prop}
\begin{proof}
For  part \eqref{Reduction1}, we must check that our construction defines a principal $\wt{G}$-bundle with connection.  Since $P \xrightarrow{f} M \xleftarrow{\pi_M} Q$ are equivariant maps between $\wt{G}$-manifolds, the pullback $f^*Q$ is naturally a $\wt{G}$-manifold.  The freeness of the $\wt{G}$-action follows easily by the following argument.  Suppose $\wt{g}\in \wt{G}$ satisfies $\wt{g} (p,q) = (p,q)$ for some point in $f^*Q \subset P \times Q$; then $gp =p \in P$ implies $g=1$ and hence $\wt{g} \in K$.  Since $K$ acts freely on $Q$, then $\wt{g}=1$.  Hence, $f^*Q \to X$ is a principal $\wt{G}$-bundle.  To show that 
\begin{equation}\label{eq:newconnection}  \Theta_{f^*Q} = \pi_P^*\Theta_P \oplus \big(\wt{f}^*\Theta_Q - \iota_{\pi_P^* \Theta_P}\wt{f}^*\Theta_Q \big) \in \O^1(f^*Q)\otimes (\fg \oplus \fk)\end{equation}
is a connection, we must check it is $\wt{G}$-equivariant and restricts fiberwise to the Maurer--Cartan 1-form on $\wt{G}$.  By Lemma \ref{lem:equivConn}, we see that $\wt{f}^*\Theta_Q - \iota_{\pi_P^* \Theta_P}\wt{f}^*\Theta_Q$ must be $\wt{G}$-invariant and $\fg$-horizontal, since it is the image of $\Theta_G$ under the map
\[ \big(\W(\fg)\otimes \O(Q) \otimes \fk \big)^{\wt{G}}_{\fg \text{-}\hor} \longrightarrow \big( \O(f^*Q) \otimes \fk \big)^{\wt{G}}_{\fg\text{-}\hor}.  \]
This, combined with the fact that $\Theta_P \in \big(\O^1(P) \otimes \fg \big)^G$, implies that \eqref{eq:newconnection} is an element of $\big( \O^1(f^*Q) \otimes(\fg \oplus \fk)\big)^{\wt{G}}.$  Similarly, if $\xi_1 \oplus \xi_2 \in \fg \oplus \fk$, then
\begin{align*}
\iota_{\xi_1 \oplus \xi_2} &\big( \pi_P^*\Theta_P \oplus (\wt{f}^*\Theta_Q - \iota_{\pi_P^* \Theta_P}\wt{f}^*\Theta_Q) \big)  \\
&= \big( \iota_{\xi_1} \Theta_P \big) \oplus \big( \iota_{\xi_2} \Theta_Q \big) + 0\oplus \iota_{\xi_1}\big( \wt{f}^*\Theta_Q - \iota_{\pi_P^* \Theta_P}\wt{f}^*\Theta_Q\big) \\
&= \xi_1 \oplus \xi_2.
\end{align*}
Therefore, the morphism of stacks $\Theta_G$ is well-defined.  The map $\pi_P$ gives a natural isomorphism $(f^*Q)/K \iso P$ compatible with the connections.  Hence, $q_{/K} \o \Theta_G$ is naturally isomorphic to the identity on $\En G \times_G M$.  

The proof of part \eqref{Reduction2} is given by tracing through the Freed--Hopkins isomorphism, discussed in Section \ref{subsec:FreedHopkins}, and the definition of our map.  Essentially, we must show that the diagram 
\[ \xymatrix{ \O_{G\times K}(Q) \ar[d]^{\Theta_G^*} \ar[r]^-{\iso}& \O\left(\En (G\times K) \times_{G\times K} Q \right) \ar[d]^{\Theta_G^*} \\
\O_G(M) \ar[r]^-{\iso} & \O(\En G \times_G M)
}\]
commutes by evaluating on a test manifold $X$.  Consider a homogeneous element $\alpha  \beta  \gamma \in \O_{G\times K}(Q) \subset \W(\fg)\otimes \W(\fk) \otimes \O(Q)$.  We first check the clockwise direction.  Given $X \to \En G \times_G M$, we have 
\[ \vcenter{\xymatrix{ (P,\Theta_P) \ar[d] \ar[r]^f &M \\ X
} }  \quad \mapsto \quad \vcenter{ \xymatrix{ (f^*Q, \Theta_{f^*Q}) \ar[r]^{\wt f} \ar[d]^{\pi_P} & (Q,\Theta) \ar[d]\\ (P,\Theta_P) \ar[r]^{f} \ar[d] & M \\ X
} }\]
where 
\[ \Theta_{f^*Q} =  \pi_P^* \Theta_P \oplus \left( \wt f^*\Theta - \iota_{\pi_P^*\Theta_P} \wt f^*\Theta \right).\]
Therefore, $\alpha \beta \gamma \in \O_{G\times K}(Q)$ gets mapped to
\[ \alpha(\pi_P^*\Theta_P) \wedge \beta( \wt f^*\Theta - \iota_{\pi_P^*\Omega_P} \wt f^*\Theta) \wedge \wt f^* \gamma \in \O(f^*Q)_{G\times K \text{-} \basic} \iso \O(X).\]
Going counter-clockwise, we first have that $\alpha \beta\gamma \mapsto \alpha \beta(\Theta_G)\gamma$, which lives in \[\left( \W(\fg) \otimes \O(Q) \right)_{G\times K \text{-} \basic } \iso \left(\W(\fg) \otimes \O(M) \right)_{G \text{-} \basic}.\]
When we evaluate on $X \leftarrow (P,\Theta) \xrightarrow{f} M$, 
\[ \alpha \beta(\Theta_G) \gamma \quad \longmapsto \quad \alpha(\Theta_P) \beta(f^*\Theta - \iota_{\Theta_P} f^*\Theta) f^*\gamma \, \in \,  \O(P)_{G\text{-}\basic} \iso \O(X).\]
Pulling this back up to $f^*Q$ via $\pi_P^*$ gives us the form from the other direction.  Hence, they restrict to the same form in $\O(X)$.

Finally, part \eqref{Reduction1} implies that $\Theta_G^* \, \o\,  q_K^* = \Id$, and we know from \eqref{eq:QuotientFree} that $q_{/K}^*$ is an isomorphism in ordinary cohomology.  Therefore, part \eqref{Reduction3} follows immediately.
\end{proof}

\subsection{Uniqueness of $\HG^*$}\label{subsec:Uniqueness}

A natural question is whether one can give an axiomatic characterization of the functor $\HG^*$.  We now show that the character diagram, combined  with the maps $\HG^*(P) \xrightarrow{\Theta^*} \HG[]^*(M)$ for principal $G$-bundles with connection, uniquely characterize our differential extension of Borel equivariant cohomology.  The idea is to use $\Theta^*$ to regard $\HG^*(M)$ a subgroup of $\HG[]^*(EG\times_G M)$ and invoke the Simons--Sullivan axiomatic characterization of differential cohomology.  It is unknown to the author if the uniqueness result holds without assuming the additional structure maps $\Theta^*$.

\begin{prop}\label{prop:Uniqueness}If $\bH^*_G$ is a functor from $G$-manifolds to graded abelian groups satisfying parts \ref{thm1diag} and \ref{thm1conn} of Theorem \ref{thm:Properties}, then there is a natural equivalence $\bH^*_G \to \HG^*$ that commutes with the identity map on all other functors in the character diagram.
\end{prop}
\begin{proof}
Assume that $\bH_G$ satisfies Theorem \ref{thm:Properties} parts \ref{thm1diag} and \ref{thm1conn}.  As described in Section \ref{subsec:deRhamModel}, let $(EG,\Theta_{EG}) \to BG$ be the universal bundle with connection obtained as a direct limit of smooth manifolds.  The projection map $EG \times M \to M$ is $G$-equivariant and can be combined with the universal connection to produce
\[ \bH_G^*(M) \to \bH_G^*(EG\times M) \xrightarrow{\Theta_{EG}^*} \bH^*(EG \times_G M). \]
On ordinary cohomology, this map is an isomorphism.  Since the Weil homomorphism  $\O_G(M) \to \O(EG\times_G M)$ is injective (Lemma \ref{lem:WeilEGinj}), it follows from the short exact sequence \eqref{ses2} that $\bH_G^*(M)$ is naturally an abelian subgroup of $\bH^*(EG\times_G M)$.  The same argument shows $\HG^*(M)$ is a subgroup of $\HG[]^*(EG\times_G M)$.  

In \cite{SS08a}, Simons--Sullivan show the character diagram uniquely characterizes ordinary differential cohomology, and there exists a unique natural transformation $\bH^* \xrightarrow{\iso} \HG[]^*$.  This isomorphism combines with \eqref{ses2} to give the following.
\[ \xymatrix{ 0 \ar[r] &\frac{\O^{*-1}_G(M)}{\O^{*-1}_G(M)_\Z} \ar[r] \ar@{^(->}[d] &\bH_G^*(M) \ar[r] \ar@{^(->}[d]& H^*_G(M;\Z) \ar[r] \ar[d]^{\iso}&0 \\
0 \ar[r] & \frac{\O^{*-1} (EG\times_G M)}{\O^{*-1} (EG\times_G M)_\Z} \ar[r]  & \HG[]^*(EG\times_G M) \ar[r] & H^*(EG\times_G M;\Z) \ar[r]&0 \\
0 \ar[r] & \frac{\O^{*-1}_G(M)}{\O^{*-1}_G(M)_\Z} \ar[r] \ar@{_(->}[u]& \HG^*(M) \ar[r] \ar@{_(->}[u] & H^*_G(M;\Z) \ar[r] \ar[u]_{\iso}&0
} \]
It immediately follows that $\bH_G^*(M)$ and $\HG^*(M)$ can be naturally considered as the same abelian subgroup of $\HG[]^*(EG\times_G M)$, and this gives us our desired natural isomorphism.
\end{proof}


\section{Equivariant Chern--Weil theory}\label{sec:EquivCW}  
The constructions in  Propositions  \ref{prop:GroupChange} and \ref{prop:Reduction} combine to give a geometric interpretation of equivariant Chern--Weil theory.  For simplicity, let $\wt{G}=G\times K$ in the following preliminary discussion.

Suppose that $M$ is a $G$-manifold and $(Q,\Theta) \to M$ is a $G$-equivariant principal $K$-bundle with $G$-invariant connection; i.e.  $(Q,\Theta)\in G\GBunc[K](M)$.  We briefly explain two classical ways to construct equivariant characteristic classes.  The first construction is purely topological and does not use the connection.  Consider the principal $K$-bundle $EG \times_G Q \to EG \times_G M$.  The choice of classifying map 
\begin{equation}\label{eq:ClassMapEquivBundle} EG \times_G M \longrightarrow BK.\end{equation}
induces a homomorphism
\[ H^*(BK;A) \longrightarrow H^*_G(M;A) \]
for any abelian group $A$.
The second method, due originally to Berline--Vergne  \cite{MR705039}, uses the connection to construct an equivariant differential form.  Given $\omega \in (S^n \fk^*)^K \iso H^{2n}(BK;\R)$, define the equivariant Chern--Weil form
\[ \omega(\Theta_G) \= \omega(\Omega_G^{\wedge n}) \in \O_G^{2n}(M). \]
As noted in Remark \ref{Rem:CartanCurv}, this becomes
\[ \omega\big( (\Omega - \iota_{\varOmega_\fg } \Theta)^{\wedge n} \big) \in \left( S \fg^* \otimes \O(M) \right)^G\]
in the more commonly used Cartan model.  While the form $\omega(\Theta_G)$ depend on the connection $\Theta$, its class in $H_G^{2n}(M;\R)$ does not.  These give equivariant Chern classes and forms when $K=U(n)$, and equivariant Pontryagin classes and forms when $K=SO(n)$.  In \cite{MR1850463}, Bott--Tu use the universal connection on $EG$ to show these two constructions give the same classes in $H^*_G(M;\R)$.  The following construction is similar to theirs,  but we replace $EG$ with $\En G$ and view everything in the world of stacks.  

Suppose that $(Q,\Theta) \in G\GBunc[K](M)$.  The following maps compose to give a refinement of \eqref{eq:ClassMapEquivBundle} to stacks.
\begin{equation}\label{eq:ChernWeil} \En G \times_G M \xrightarrow{\Theta_G} \En (G\times K) \times_{G\times K} Q \longrightarrow \Bn (G \times K)  \xrightarrow{q_{/G}} \Bn K.\end{equation}
The first map is given by the bundle pullback construction of Section \ref{subsec:Reduction}, the second is induced by the equivariant map $Q\to \pt$, and the third map is induced by the quotient $G\times K \xrightarrow{/G} K$.  When applied to a test manifold $X$, the composition performs the following construction.
\[  \vcenter{ \xymatrix{& (Q,\Theta_Q) \ar[d]^{\pi_M} \\ (P, \Theta_P) \ar[r]^{f} \ar[d]^{\pi_X} & M \\X }} \longmapsto  \vcenter{ \xymatrix{ \big(f^*Q, \pi_P^* \Theta_P \oplus (\wt f^* \Theta_Q - \iota_{\Theta_P} \wt{f}^* \Theta_Q) \big) \big/ G \ar[d] \\ X} } \]
It pulls back $Q$, modifies the connection to create a $G\times K$-bundle with connection, and then quotients by $G$.

\begin{thm}\label{thm:equivCCS}Associated to $(Q,\Theta) \in G\GBunc[K](M)$ is 
a natural morphism
\[ \En G \times_G M \xrightarrow{(Q_G,\Theta_G)} \Bn K, \]
such that the induced map 
\[ \HG[K]^* \xrightarrow{(Q_G,\Theta_G)^*} \HG^*(M)\] 
refines the traditional equivariant characteristic classes and forms as described above.
\end{thm}
\begin{proof}We know the morphism $\En G \times_G M \to \Bn K$ is well-defined, since it is defined in \eqref{eq:ChernWeil} as the composition of three morphisms.

We now must show that at the level of cohomology, $(Q_G,\Theta_G)^*$ is the homomorphism $\psi^*$, where $EG\times_G M \xrightarrow{\psi} BK$ is a classifying map for $EG\times_G Q \to EG\times_G M$.  By Proposition \ref{prop:StackCohIso}, the maps $EG\times_G M \to \En G \times_G M \to \E G \times_G M$ and $BK \to \Bn K \to \B K$  induce isomorphisms in cohomology.  Therefore, we must check that the diagram
\begin{equation}\label{eq:equivCCSdiag} \vcenter{ \xymatrix{ EG\times_G M \ar[d]_{\psi} \ar[r]  & \E G \times_G M \ar[d]^{Q_G} \\ BK \ar[r]  & \B K,} } \end{equation}
where the two horizontal morphisms given by viewing $EG\to BG$ and $EK \to BK$ as principal bundles, is commutative up to isomorphism.

We evaluate on a test manifold $X$ and trace through where $X\xrightarrow{s} EG\times_G M$ is sent, using the following diagram to identify maps.
\[ \xymatrix{ \wt{s}^*(\Pi_M^* Q) \ar[r]\ar[d] & \Pi_M^*Q \ar[r] \ar[d] & Q \ar[d] \\
s^*(EG\times M) \ar[d] \ar[r]^{\wt{s}} & EG \times M \ar[d] \ar[r]^-{\Pi_M} & M \\
X \ar[r]^{s} &EG\times_G M 
}\]
Going clockwise in \eqref{eq:equivCCSdiag}, we have
\[  X\xrightarrow{s} EG\times_G M  \;  \mapsto \;  \vcenter{ \xymatrix{s^*(EG\times M) \ar[r]^(.65){\Pi_M \o \wt{s} } \ar[d] & M \\ X} } \; \mapsto \; \vcenter{ \xymatrix{ \big((\Pi_M \o \wt{s})^*Q\big)/G \ar[d] \\ X } }, \]
and the evident isomorphism $\Pi_M^*Q \iso EG\times Q$ gives 
\[ \big( (\Pi_M \o \wt{s})^*Q \big)/G \iso \big( \wt{s}^*(EG\times Q) \big)/G  \iso s^* (EG\times_G Q). \]
Going counter-clockwise in \eqref{eq:equivCCSdiag}, we have
\[  X \xrightarrow{s} EG\times_G M   \; \mapsto \;   X \xrightarrow{\psi \o s} BK   \; \mapsto \;   \vcenter{ \xymatrix{ (\psi \o s)^* EK \ar[d] \\ X }}. \]
Since
\[ (\psi \o s)^* EK \iso s^*(\psi^*EK) \iso s^*(EG\times_G Q),\]
the diagram  \eqref{eq:equivCCSdiag} commutes up to isomorphism.

For differential forms, we trace through the composition \eqref{eq:ChernWeil}, which is given by the restriction of the following to the basic sub-complexes
\[ \W(\fk) \to \W(\fg) \otimes \W(\fk) \to \W(\fg)\otimes \W(\fk) \otimes \O(Q) \xrightarrow{\Theta^*_G} \W(\fg)\otimes \O(M).  \]
The first two maps are given by the mapping $\W(\fk)$ into its factor in the tensor product, so the composition is simply
\begin{align*}
\left( S \fk^*\right)^K = \W(\fk)_{\basic}  &\longrightarrow \left( \W(\fg) \otimes \O(M) \right)_{\basic} = \O_G(M) \\
\omega &\longmapsto \omega(\Theta_G),
\end{align*}
which is precisely the equivariant Chern--Weil form in the Weil model.
\end{proof}

One can also construct equivariant extensions of the Chern--Simons forms.  As explained in \eqref{eq:ChernSimonsWeil}, to $\omega \in (S^k\fk^*)^K$ is naturally associated a Chern--Simons form $\CS_\omega \in \W^{2k-1}(\fk)^K$ satisfying $d_\W \CS_\omega = \omega$.  Given $(Q,\Theta) \in G\GBunc[K](M)$, we define the equivariant Chern--Simons form to be the image of $\CS_\omega$ under the equivariant equivariant Weil homomorphism 
\begin{align}\label{eq:EquivChernSimonsForm}\begin{split}
\W^{2k-1} &\xrightarrow{\Theta_G^*} \O^{2k-1}_G(Q) \\
\CS_\omega &\longmapsto \CS_\omega(\Theta_G).  \end{split}
\end{align}
As in the case of differential characters, the Chern--Simons forms are closely related to the differential refinements of characteristic classes.  

\begin{prop}\label{prop:EquivChernSimons}Let $\check{\omega} \in \HG[]^{2k}(\Bn K) \iso H^{2k}(BK;\Z)$ be a differential refinement of $\omega \in (S^k \fk^*)^K$.  Suppose that $(Q,\Theta) \in G\GBunc[K](M)$ admits a $G$-equivariant section $s:M \to Q$.  Then 
\[ [ s_G^* \CS_\omega(\Theta_G) ] = \check{\omega}(\Theta_G) \in \HG^{2k}(M) .\]
\end{prop}
\begin{proof}
First note that because $EG$ is contractible, the character diagram for $\En G$ gives isomorphisms
\[ \xymatrix@R=3mm@C=10mm{& \HG[]^{2k}(\En G) \ar[dr]^{\iso} \\
{\displaystyle \frac{\W^{2k-1}(\fk)}{d_{\W}  \W^{2k-2}(\fk) }} \ar[ur]^{\iso} \ar_{d_{\W}}[rr]&& \W^{2k}(\fk)_0.
}\]
This gives the universal property $\pi^*\check{\omega} = [\CS_\omega] \in \HG[]^{2k}(\En K)$.

The construction \eqref{eq:ChernWeil} readily generalizes to the following diagram,
\[ \xymatrix@C=1.5cm{ \En G \times_G Q \ar[r] \ar[d] &\En K \ar[d]^{\pi} \\
\En G \times_G M \ar[r]^-{(Q_G,\Theta_G)} \ar@/_1pc/[u]_{s_G}  & \Bn K }\]
giving us two ways of writing the morphism $\En G \times_G M \to \Bn K$.  The bottom morphism induces $\check{\omega}(\Theta_G)$; the other one gives
\[\xymatrix@R=0mm{ \HG[](\Bn K) \ar[r] & \HG[](\En G) \ar[r] & \HG[](\En G\times_G Q) \ar[r] &  \HG[](\En G \times_G M) \\
\check{\omega} \ar@{|->}[r] & [\CS_\omega]  \ar@{|->}[r] & [\CS_\omega(\Theta_G)]  \ar@{|->}[r] & [s_G^*\CS_\omega(\Theta_G)] .
} \]
\end{proof}


\begin{rem}In the statement of Theorem \ref{thm:equivCCS}, we assumed that the total space of the bundle $Q$ had an action of $G\times K$.  Suppose instead that $(Q,\Theta) \in G\GBunc[K](M)$, where $\wt{G}/K\iso G$ with a fixed Lie algebra splitting \eqref{eq:KGassumption}.  It no longer makes sense to quotient by $G$ at the end, but we still have the composition
\[ \En G \times_G M \xrightarrow{\Theta_G} \En \wt{G} \times_{\wt{G}} Q \longrightarrow \Bn \wt{G}.\]
This induces differential equivariant characteristic classes via the homomorphism
\[ \HG[\wt{G}]^* \longrightarrow \HG^*(M).\]
Note that Proposition \ref{prop:EquivChernSimons} would not be applicable unless there exists an isomorphism $\wt{G} \iso G \times K$.  
\end{rem}

\subsection{Equivariant differential characters}\label{subsec:equivDC}
One convenient description of differential cohomology is via differential characters \cite{CS85}.  An element $\x \in \HG[]^n(M)$, with curvature $\curv(\x) = \omega$, associates to every smooth singular $(n-1)$-cycle $X_{n-1} \xrightarrow{z} M$ an element  $\langle \x, z \rangle \in \R/\Z$; if $z=\partial \wt{z}$ for some smooth $n$-chain $Y_n \xrightarrow{\wt{z}} M$, then
\[ \langle \x, z \rangle = \int_{Y_n} \wt{z}^{^*} \omega \mod \Z.\]
In degree 2, an $\R/\Z$-bundle with connection on $M$ naturally defines an element of $\HG[]^2(M)$ via its holonomy.  More generally, evaluating a differential character on a $(n-1)$-cycle can be thought of as the holonomy for some higher gerbe or abelian gauge field.  

This idea generalizes to the equivariant setting, with classes in $\HG^n(M)$ giving ``differential characters'' on the stack $\En G \times_G M$.  This type of structure naturally appears in the physics literature as WZW terms for gauged sigma models \cite{MR1311654}.  We state the main idea here, but we do not give enough details on singular cycles to consider this a definition of $\HG^*(M)$.

Assume that $X_{n-1}$ is a closed smooth manifold.  To any map 
\[ X_{n-1} \xrightarrow{z} \En G \times_G M,\]
a class $\x \in \HG^n(M)$ associates an ``equivariant holonomy'' in $\R/\Z$, as seen in the following diagram.  
\begin{align*}
\vcenter{ \xymatrix{(P,\Theta) \ar[r]^f \ar[d] & M \\ X_{n-1}} }  \quad &\Longrightarrow \quad \vcenter{\xymatrix{ \HG^n(P) \ar[d]^{\Theta^*} & \HG^n(M) \ar[l]_{f_G^*} \\ \HG[]^n(X_{n-1}) \iso \R/\Z}  }\end{align*}
If the cycle is a boundary, with $[(P,\Theta) \to X_{n-1}] = \partial\left[ (P',\Theta') \to Y_n \right]$ and $f$ extending to a map $P' \xrightarrow{\wt{f}} M$, then the holonomy can be computed by integrating the equivariant curvature $\omega = \curv(\x) \in \O^n_G(M)$
\[ \langle \x, \; X_{n-1} \leftarrow (P,\Theta) \xrightarrow{f} M \rangle = \int_{Y_n} \left( \Theta^* \otimes \wt{f}^* \right) \omega \mod \Z.\]
Furthermore, these equivariant holonomies are gauge-invariant; if two objects in the groupoid $\Map(X_{n-1}, \En G\times_G M)$ are isomorphic, they associate to $\x \in \HG^n(M)$ equal elements in $\R/\Z$.

\begin{rem}The entire above discussion concerning equivariant characters immediately generalizes from $\Z \subset \R$ to $\L \subset V$, where elements of $\HG^*(M;\L)$ give characters valued in $V/\L$.
\end{rem}

\subsection{Geometric models in low degrees}

\begin{prop}There are natural isomorphisms
\begin{align*}
 \HG^0(M) &\iso H^0(M/G;\Z), \\
 \HG^1(M) &\iso \HG[]^1(M)^G \iso C^\infty(M, \R/\Z)^G.
\end{align*}
\end{prop}
\begin{proof}
The short exact sequence \eqref{ses1} implies that $\HG^0(M) \iso  \O^0_G (M)_\Z.$  Unraveling this, we obtain the following chain of equalities and isomorphisms,
\[ \HG^0(M) \iso \O_G^0(M)_\Z = C^\infty(M,\Z)^G \iso C^0(M/G,\Z) \iso H^0(M/G;\Z).\]

In degree one, consider the homomorphism $\HG^1(M) \to C^\infty(M,\R/\Z)^G$ defined as follows using the equivariant differential character construction of Section \ref{subsec:equivDC}.  Given $\x \in \HG^1(M)$ and a point $m\in M$, define the element $F_{\x}(M) \in \R/\Z$ by evaluating on the $G$-orbit of $m$,
\[ F_{\x}(m) \=  \left\langle  \x, \; \pt \leftarrow G \xrightarrow{G\cdot m} M \right\rangle.\]
The gauge invariance of the equivariant character implies that the function $F_{\x}$ is $G$-invariant.  Furthermore, $\O^1_G(M)_\Z \iso \O^1(M)^G_\Z$, and integration shows that $F_{\x}$ is a smooth function with derivative $\curv(\x)$.  Hence, $F_{\x} \in C^\infty(M,\R/\Z)^G$.

Note that the kernel of $C^\infty(M,\R/\Z)^G \xrightarrow{d} \O^1(M)_\Z^G$ is given by the constant functions
\[ \Ker(d) = \O^0(M;\R/\Z)^G_{\cl} \iso H^0(M/G;\R/\Z) \iso H^0_G(M;\R/\Z).\]
Therefore, the homomorphism of short exact sequences 
 \[ \xymatrix{
0 \ar[r] & H^0_G(M;\R/\Z) \ar[r] \ar[d]^{\iso} &\HG^1(M) \ar[r] \ar[d] & \O_G^1(M)_\Z \ar[d]^{\iso} \ar[r] & 0 \\
0\ar[r] & \Ker(d) \ar[r]  & C^\infty(M, \R/\Z)^G \ar[r] &\O^1(M)^G_\Z  \ar[r] &0 \\
}  \]
gives our desired isomorphism via the Five Lemma.
\end{proof}

\begin{prop}In degree two, there is a natural isomorphism
\begin{align*}
 \HG^2(M) &\iso \pi_0\big(G\GBunc[\R/\Z](M)\big);
\end{align*}
i.e. $G$-equivariant $\R/\Z$-bundles with connection on $M$, modulo equivariant connection-preserving isomorphisms, are in bijection with elements of $\HG^2(M)$.
\end{prop}

\begin{proof}
%
%

Let $\D c \in \HG[\R/\Z]^2 \iso H^2(B \R/\Z;\Z) \iso \Z$ be the standard generator.  The equivariant Chern--Weil construction of Theorem \ref{thm:equivCCS} gives us a map
\begin{align} \label{eq:EquivLineBundles} \begin{split} G\GBunc[\R/\Z](M) &\longrightarrow \HG^2(M) \\
(Q, \Theta) &\longmapsto \D c(Q_G, \Theta_G). \end{split}
\end{align}
Multiplication in $\R/\Z$ makes $G\GBunc[\R/\Z](M)$ into a Picard groupoid (this symmetric monoidal structure corresponds to the tensor product of line bundles), and \eqref{eq:EquivLineBundles} is symmetric monoidal.  Furthermore, the map only depends on the isomorphism class of $(Q,\Theta)$ in $G\GBunc[\R/\Z](M)$.

Hence, \eqref{eq:EquivLineBundles} descends to an abelian group homomorphism $\pi_0 \left( G\GBunc[\R/\Z](M) \right) \to \HG^2(M).$  Furthermore, the characteristic class $c(Q_G) \in H^2_G(M;\Z)$ does not depend on the connection $\Theta$.  This gives the following homomorphism of short exact sequences.
\[ \xymatrix{ 0 \ar[r]& \ar[r] \ar[d] \Ker \Pi \ar[r] \ar[d] & \pi_0 \left[ G\GBunc[\R/\Z](M) \right]\ar[r]^-{\Pi} \ar[d]& \pi_0 \left[ G\GBun[\R/\Z](M) \right] \ar[r] \ar[d]& 0 \\
0 \ar[r] & \ar[r] \frac{\O^1_G(M)}{\O^1_G(M)_\Z} \ar[r]& \HG^2(M) \ar[r]& H^2_G(M;\Z) \ar[r]& 0} \]
We now show that each of the above vertical maps is an isomorphism.

It is a classical result \cite{MR711050} that equivariant bundles with abelian structure group are classified topologically by Borel cohomology, i.e. $\pi_0 \big( G\GBun[A](M) \big) \iso [ EG\times_G M, BA]$ when $A$ is abelian.  This implies the right vertical map is an isomorphism.

The $G$-invariant connections on the trivial bundle $M \times \R/\Z$ are in bijection with $\O^1(M)^G$.  Such a connection form is gauge equivalent to the trivial connection if and only if it is the derivative of a gauge transformation $C^\infty(M,\R/\Z)^G \iso \HG^1(M)$; i.e. if it lives in $\O^1(M)^G_\Z$.  Hence, 
\[ \Ker \Pi \iso  \frac{\O^1(M)^G}{\O^1(M)^G_\Z} \xrightarrow{\iso} \frac{\O^1_G(M)}{\O^1_G(M)_\Z},\]
giving that the left vertical arrow is an isomorphism.  The proposition now follows by the Five Lemma.
\end{proof}

\section{Details of simplicial sheaf construction}\label{Section:DefnProof}

This section contains a more detailed account of Section \ref{subsec:sPreDefn}.  The results essentially follow by combining the work of \cite{MR3049871} and \cite{1311.3188}; the reader may refer to these works, along with \cite[Chapter 5]{MR2522659} for further details.  Let us briefly explain our notation.  

We will use simplicial sets in order to deal with groupoids, non-negatively graded chain complexes, and topological spaces simultaneously.  Let $\Gpdinfty$ be the $(\infty,1)$-category of ``$\infty$-groupoids,'' taken here to be the full simplicially enriched subcategory of simplicial sets spanned by Kan complexes.  The category $\Gpdinfty$ is naturally equivalent, as an $(\infty,1)$-category, to the topologically enriched category $\Top$ of topological spaces with the homotopy type of a CW complex.  Though it is not strictly necessary, we use this equivalence $\Gpdinfty \simeq \Top$ to make certain statements easier to read.
Any groupoid is naturally a Kan complex via the nerve construction $\N$, and the Dold--Kan correspondence $\DK$ makes a non-negatively graded chain complex into a simplicial abelian group.  These structures, together with the singular functor $\S$ and geometric realization $|\,|$, fit into the following picture.
\[ \xymatrix{  & \Top  \ar@<1.2ex>[d]^{\S}_{  \simeq }\\
\Gpd \ar@{^`->}[r]^{\N} &\Gpdinfty \ar@<1.2ex>[u]^{| \,|}& \Ch_{\geq 0} \ar@{_`->}[l]_{\DK} } \]
We implicitly consider any ordinary category as an $(\infty, 1)$-category via the nerve construction, though we do not explicitly write it.

Let $\PShv_{\Gpdinfty} = \Fun(\Man^{\op}, \Gpdinfty)$ be the $(\infty, 1)$-category of $\Gpdinfty$-valued presheaves on manifolds.  A functor $\cF \in \PShv_{\Gpdinfty}$ is a sheaf if it satisfies the following descent condition: for any covering $\cU \to X$, the canonical map $\cF(X) \xrightarrow{\simeq} \lim_{\Delta} \cF(\cU^\bullet)$ is an equivalence.  The full subcategory of simplicial sheaves is denoted $\Shv_{\Gpdinfty}$, and there is a sheafification functor $L$ forming the adjunction
\[ L: \PShv_{\Gpdinfty} \rightleftarrows \Shv_{\Gpdinfty} : \text{inclusion}. \] 

We continue to suppress the embedding of $\Set \into \Gpd$, but we use the nerve symbol to denote $\N\colon  \Shv_{\Gpd}\to \Shv_{\Gpdinfty}$.  As in previous sections, we will use $M$ to denote both a manifold $M$ and the stack it represents, but we will begin using the notation $\uline{M}$ within proofs for added clarity.  For $Y \in \Top$, we define $\Sing Y \in \Shv_{\Gpdinfty}$ by $\Sing Y(X) = \S(Y^X)$, the singular complex of the space of continuous maps from $X\to Y$.  Let $\cK(A,n)\in \Shv_{\Gpdinfty}$ denote any simplicial sheaf equivalent equivalent to $\Sing K(A,n)$.


We now recount the results we need from the theory of homotopy-invariant sheaves developed by Bunke--Nikolaus--V\"olkl in \cite{1311.3188}.  Their results hold for general $\cC$-valued sheaves on manifolds, where $\cC$ is an $(\infty, 1)$-category, but we have specialized them to the case of $\cC = \Gpdinfty \simeq \Top$.  A sheaf $\cF$ is homotopy-invariant if, for all manifolds $X$, the projection $X \times I \to X$ induces an equivalence
\[ \cF(X) \xrightarrow{\simeq} \cF(X \times I).\]
Let  $\Shv^{\h}_{\Gpdinfty}$ denote the full subcategory of homotopy-invariant sheaves.  The following results are found in Proposition 2.6, equation (63), and Lemma 7.13 of \cite{1311.3188}.
\begin{itemize}
\item The functor $\Sing$ gives an equivalence $\Top \xrightarrow{\simeq} \Shv_{\Gpdinfty}^{\h}$;  the inverse is given by evaluating at a point and taking the geometric realization.  There is also a homotopification functor $\cH$, and these form the following adjunctions.
\[ \xymatrix{ \Shv_{\Gpdinfty} \ar@<.75ex>[r]^{\cH} &\Shv_{\Gpdinfty}^{\h}  \ar@<.75ex>[r]^-{\text{ev}(\pt)}_-{\sim}  \ar@<.75ex>@{_`->}[l]& \Top \ar@<.75ex>[l]^-{\Sing}} \]
\item These adjunctions combine to give the following adjunction:
\begin{equation}\label{eq:HtpyAdjunction} \h: \Shv_{\Gpdinfty} \rightleftarrows \Top : \Sing .
\end{equation}
Because $\h$ is a left adjoint, it automatically commutes with homotopy colimits.  The functor $\h$ also preserves finite products.
\item Let $M$ be a manifold and $\uline{M}\in \Shv_{\Gpd}$ the stack it represents.  Then, 
\begin{equation}\label{eq:HtpyYoneda} \h( \N(\uline{M}) ) \simeq  M \in \Top.\end{equation}
\item If $\cA \in \Shv_{\Ab} \into \Shv_{\Gpd}$ is a sheaf of abelian groups that are $C^\infty$-modules, then  \begin{equation}\label{eq:FineSheaf} \h( \N(\cA)) \simeq \pt \in \Top. \end{equation}
\end{itemize}

We now make the following definition/abbreviation.
\begin{defn}\label{defn:StackCoh}Let $A$ be a fixed abelian group.  For a stack $\cM \in \Shv_{\Gpd}$,
\[ H^n(\cM; A) \=  \ho \Shv_{\Gpdinfty}(\N(\cM), \cK(A,n) ).\]
\end{defn}

\begin{example}For $M \in \Man$, consider $\N(\uline{M}) \in \Shv_{\Gpdinfty}$, which is defined by considering the set $C^\infty(X, M)$ as a discrete simplicial set.  The Yoneda lemma shows that Definition \ref{defn:StackCoh} agrees with our ordinary definition of cohomology,
\[ H^n( \uline{M}, A) = \ho \Shv_{\Gpdinfty}(\N(\uline{M}), \Sing K(A,n)) \iso \ho \Top(M,K(A,n)) \iso H^n(M;A).\]
\end{example}

The following proposition and proof are taken almost directly from Lemma 5.2 in \cite{1311.3188}.

\begin{prop}\label{prop:EnGHtpy}The natural map $EG\times_G M \to \En G \times_G M$ induces equivalences
\[ EG \times_G M \xrightarrow{\simeq} \h(\N(\En G \times_G M)) \xrightarrow{\simeq} \h(\N(\E G \times_G M)). \]
\end{prop}
\begin{proof}
We first show $\h(\N(\En G \times_G M)) \xrightarrow{\simeq} \h(\N( \E G \times_G M))$.  Consider the following commutative diagram of prestacks
\[ \xymatrix{  \En G\times_G M \ar[r] & \E G \times_G M \\
 \left(\O^1 \otimes \fg \times M\right)\sslash G  \ar[u] \ar[r] & M \sslash G,  \ar[u] 
}\]
where we use the notation $\bullet \sslash G$ to denote the action groupoid $[\bullet \leftleftarrows G \times \bullet]$ associated to a $G$-action.  The group $G$ acts on $\O^1 \otimes \fg$ via the usual change of connection formula, $\omega \cdot g = \Ad_{g^{-1}} \omega + g^*\theta_{\fg} $.  The right vertical maps is given by associating to $X \to (M\sslash G)$ the trivial bundle $X \leftarrow X\times G$, together with the map $X\times G \to M$ determined by the map $X\to M$.  Likewise, the left vertical map is given by the same construction, but with the connection on $X\times G$ determined by the map $X\to \O^1\otimes \fg$.  

Since any bundle is locally trivializable, these vertical maps induce equivalences of groupoids when evaluated on stalks.  Hence, upon taking the nerve, the vertical maps induce equivalences in $\Shv_{\Gpdinfty}$ (see \cite[Lemma 5.2]{1311.3188} and \cite[(7.23)]{MR3049871}).
\[ \xymatrix{  \N(\En G\times_G M )\ar[r] & \N(\E G \times_G M) \\
L( \N\left( \left(\O^1 \otimes \fg \times M\right)\sslash G \right))  \ar[u]^{\simeq} \ar[r] & L(\N \left(M \sslash G\right))  \ar[u]^{\simeq} 
}\]

We now use the general fact that for a simplicial set $B$, the induced simplicial object in simplicial sets, given by the composition $\Delta^{\op} \xrightarrow{B} \Set \xrightarrow{\N} \sSet$, gives a natural equivalence
\[ \hocolim_{\Delta^{\op}} \N(B) \simeq B.\] 
Rewriting the bottom entries as colimits, and applying the properties of $\h$ described above, gives the following sequence of equivalences. 
\begin{align*}
\h \big( L( \N &((\O^1 \otimes \fg \times M ) \sslash G  ))\big) \simeq 
\h \left(L( \O^1 \otimes \fg \times \uline{M} \leftleftarrows \uline{G}\times \O^1 \otimes \fg \times \uline{M} \tripleleftarrow  \cdots )\right) \\
&\simeq \h\Big(L \Big( \hocolim_{\Delta^{\op}}\big( \N(\O^1 \otimes \fg\times \uline{M}) \leftleftarrows \N(\uline{G}\times \O^1 \otimes \fg\times \uline{M}) \tripleleftarrow  \cdots \big)\Big)\Big) \\
&\simeq \hocolim_{\Delta^{\op}} \left( \h(\N(\O^1\otimes \fg) \times \h(\N(\uline{M})) \leftleftarrows \h(\N(\uline{G})) \times \h(\N(\O^1\otimes \fg)) \times \h(\N(\uline{M})) \tripleleftarrow   \cdots )\right) \\
&\xrightarrow{\simeq} \hocolim_{\Delta^{\op}} \left( M \leftleftarrows G \times M \tripleleftarrow G \times G \times M \cdots \right) \simeq \h \big(L \big( \N(M\sslash G)\big)\big).
\end{align*}
The equivalence $\h(\N(\En G \times_G M)) \xrightarrow{\simeq} \h(\N(\E G\times_G M))$ follows immediately.  In the above argument, the third equivalence was given by fact that $\h$ and $L$ commute with colimits,  $\N(\uline{G}^{\times k} \times \O^1 \otimes \fg \times \uline{M})$ is a sheaf, and $\h$ preserves finite products.  The fourth equivalence was given by the fact that $\h(\N(\uline{X})) \simeq X$, and $\h(\N(\cA)) \simeq \pt$ for $\cA$ a sheaf of $C^\infty$-modules.  Note that the middle entries were given as the homotopy colimit of a simplicial space (simplicial object in $\Top$).  For proper simplicial spaces, the homotopy colimit is equivalent to the geometric realization; this is a common construction of $BG$ when $M=\pt$.


Now consider also the following commutative diagram of prestacks
\[ \xymatrix{  EG\times_G M \ar[r] & \E G \times_G M \\
(EG \times M) \sslash G  \ar[u] & M \sslash G, \ar[u] \ar[l]
}\]
where the bottom left map is given by the homotopy equivalence $M \xrightarrow{\simeq} EG\times M$.  The left vertical map is given by quotienting, and since $G$ acts freely, it leads to an equivalence of groupoids when evaluated on any manifold $X$.  Hence, both vertical maps give equivalences in $\Shv_{\Gpdinfty}$.  
\[ \xymatrix{ \N( EG\times_G M) \ar[r] & \N(\E G \times_G M) \\
L(\N[EG \times M \leftleftarrows G \times (EG \times M) ]) \ar[u]^{\simeq} &L(\N [M \leftleftarrows G \times M]) \ar[u]^{\simeq} \ar[l]
}\]
By the same calculation  as above
\begin{align*} \h(L(\N(M\sslash G))) &\simeq \hocolim_{\Delta^{\op}} \left( M \leftleftarrows G \times M \tripleleftarrow G \times G \times M \cdots \right) , \\
 \h(L(\N((EG \times M)\sslash G)))  &\simeq \hocolim_{\Delta^{\op}} \left( EG \times M \leftleftarrows G \times EG \times M \tripleleftarrow G \times G \times M \cdots \right), 
\end{align*}
and the homotopy equivalence $\pt \xrightarrow{\simeq} EG$ induces a homotopy equivalence of the relevant simplicial spaces.  This implies $\h(L(\N(M\sslash G))) \xrightarrow{\simeq}\h(L(\N((EG \times M)\sslash G)))$, which in turn gives the equivalence
\[ EG \times_G M \simeq \h(L(\N( EG\times_G M))) \xrightarrow{\simeq} \h(\N(\E G \times_G M)).\]
\end{proof}

\begin{prop}\label{prop:StackCohIso}The map $EG\times_G M \to \En G \times_G M$, defined by the connection on $EG$, induces an isomorphism in cohomology
\[ H^*(\En G \times_G M; A) \xrightarrow{\iso} H^*(EG\times_G M; A).\]
\end{prop}
\begin{proof}
This follows almost immediately from Definition \ref{defn:StackCoh}, the adjunction \ref{eq:HtpyAdjunction}, and the calculation of $\h(\N(\En G \times_G M))$ in Proposition \ref{prop:EnGHtpy}.  Together, these give the following sequence of isomorphisms:
\begin{align*} H^n( \En G \times_G M; A) &= \ho \Shv_{\Gpdinfty} ( \N( \En G \times_G M), \cK(A,n) ) \\
&\iso \ho \Top( \h( \N( \En G \times_G M)), K(A,n)) \\
& \xrightarrow{\iso} \ho \Top( EG \times_G M, K(A,n)) \\
&\iso H^n(EG\times_G M;A).
\end{align*}
\end{proof}


Finally, we repeat an explicit construction of a simplicial sheaf representing differential cohomology.  This construction, and minor variations, have already appeared in several places, including \cite{HS05} implicitly and \cite{1311.3188,1208.3961, MR3019405, MR3335251, 1310.7930} explicitly.  In order to obtain simplicial sets from cochain complexes, we use the standard trick of reversing the grading and shifting, and then using the Dold--Kan construction.  The inclusions of
 \[ \Z \into \R \into \O^0 \]
induce the following maps of presheaves of chain complexes, where degree 0 is written on the left.
\[ ( \O^n_\cl \leftarrow 0 \leftarrow \cdots )  \into ( \O^n_{\cl} \xleftarrow{d} \O^{n-1} \xleftarrow{d} \cdots \xleftarrow{d} \O^0 )  \hookleftarrow (0 \leftarrow \cdots \leftarrow \Z \leftarrow \cdots )\]
Applying Dold--Kan (and sheafifying where necessary), we then take the homotopy pullback in $\Shv_{\Gpdinfty}$, obtaining
\begin{equation}\label{eq:DiffK(Z,n)} \vcenter{ \xymatrix{ \K(\Z,n) \ar[d] \ar[r] & L(\DK(0 \leftarrow \cdots \leftarrow \Z \leftarrow \cdots )) \ar[d] \\
\DK ( \O^n_\cl \leftarrow 0 \leftarrow \cdots )   \ar[r] & \DK( \O^n_{\cl} \xleftarrow{d} \O^{n-1} \xleftarrow{d} \cdots \xleftarrow{d} \O^0 ) } }\end{equation}
The top right corner is equivalent to $\cK(\Z,n)$, and the bottom right corner is equivalent to $\cK(\R,n)$, so the construction of $\K(\Z,n)$ fits the description from \eqref{eq:SheafSquare}.


%

We now have two potential definitions of $\HG^*(M)$.  The first was given in Section \ref{sec:HS} using the cochain complex $\CG(n)(M;\Z)$, and the second is given in this section as 
\[ \ho \Shv_{\Gpdinfty}\left(\N(\En G \times_G M), \K(\Z,n)\right).\]
We now prove these two possible definitions of $\HG^*(M)$ agree.

\begin{proof}[Proof of Theorem \ref{thm:TwoDefnsAgree}]
We use Dold--Kan to consider $\CG(n)^*(M)$ as a homotopy pullback of simplicial abelian groups via the following.
\begin{equation}\label{eq:DKofHS} \vcenter{\xymatrix{\DK( \ZG(n)^n(M) \xleftarrow{d} \CG(n)^{n-1}(M) \xleftarrow{d}  ) \ar[r] \ar[d] &\DK(Z^n(EG\times_G M;Z) \xleftarrow{\delta} \cdots) \ar[d] \\
\DK(\O^n_G(M)_\cl \xleftarrow{}0) \ar[r]  & \DK(Z^n(EG\times_G M;\R) \xleftarrow{\delta} \cdots) 
} }\end{equation}

The Freed--Hopkins isomorphism \cite[Propositions 6.22 and Theorem 7.28]{MR3049871} gives us 
\[ \ho \Shv_{\Gpdinfty}(\N(\En G \times_G M), \N(\O^n_{\cl})) \iso \Shv_{\Gpd}( \En G \times_G M, \O^n_{\cl}) \iso \O^n_G(M)_\cl.\]
Likewise Proposition \ref{prop:EnGHtpy} shows us that
\begin{align*}
\Shv_{\Gpdinfty}&(\N(\En G \times_G M), L(\DK(0 \leftarrow \cdots \leftarrow \Z ) ) \simeq \Shv_{\Gpdinfty} ( \N(\En G \times_G M), \cK(\Z, n)) \simeq \\
\simeq &\Shv_{\Gpdinfty}^{\h}\left( \h(\N(\En G \times_G M)), \cK(\Z,n)   \right)  \simeq \S\left( K(\Z,n) ^{EG\times_G M } \right),  \\
\Shv_{\Gpdinfty}&(\N(\En G \times_G M), \DK( \O^n_{\cl} \xleftarrow{d} \O^{n-1} \xleftarrow{d} \cdots ) ) \simeq \Shv_{\Gpdinfty} ( \N(\En G \times_G M), \cK(\R, n)) \\ 
\simeq &\Shv_{\Gpdinfty}^{\h}\left( \h(\N(\En G \times_G M)), \cK(\R,n)   \right)  \simeq \S\left( K(\R,n) ^{EG\times_G M } \right). 
\end{align*}
There is a natural homotopy equivalence 
\[ \S(K(A,n)^X ) \xrightarrow{\simeq} \DK( Z^n(X;A) \xleftarrow{\delta} C^{n-1}(X;A) \xleftarrow{\delta} \cdots ) \]
induced by pulling back a fundamental cocycle in $Z^n(K(A,n);A)$ and integrating over the simplex \cite[Proposition A.12]{HS05}.  Therefore, we have equivalences from three corners of \eqref{eq:DiffK(Z,n)}  evaluated on $\N(\En G \times_G M)$, and the corresponding three corners of \eqref{eq:DKofHS}.  Therefore, the induced map between homotopy pullbacks 
\[ \Shv_{\Gpdinfty}(\N(\En G \times_G M), \K(\Z,n)) \xrightarrow{\sim} \DK( \ZG(n)^n \xleftarrow{d} \CG(n)^{n-1} \xleftarrow{d} \cdots ) \]
is also an equivalence in $\Gpdinfty$, which gives us the desired isomorphism in the homotopy category.
\end{proof}

\appendix
\renewcommand\K{\operatorname{K}}

\section{Equivariant de Rham theory}\label{sec:Background}
We briefly recall some standard facts about equivariant cohomology, principal bundles with connection, and equivariant differential forms.  Our goal is to describe the Weil algebra $\W(\fg)$ from the perspective of differential forms on principal bundles.  We hope this makes clear why having connections is so essential to using $\O_G(M)$.  Among the many wonderful references, our treatment is most heavily influenced by: the introduction to principal bundles in Section 1 of \cite{Fre95}, the summary of equivariant cohomology \cite{Meinrenken-EquivCoh}, and the standard textbook for $G^\star$-algebras and equivariant de Rham cohomology \cite{GuilleminSternberg99}.

\subsection{Equivariant cohomology}\label{subsec:BorelCoh}
Let $G$ be a reasonably nice topological group.  Then there exists a contractible space $EG$ on which $G$ acts freely from the right, giving us a universal $G$-bundle $EG \to BG$.  If $M$ has a continuous left $G$-action, the \textit{homotopy quotient} (or \textit{Borel construction}) of $M$ is defined
\[ EG \times_G M \= \left( EG \times M\right)\big/ \big((eg,x) \sim (e,gx)\big). \]
Borel's construction of the \textit{equivariant cohomology} of $M$, with coefficients in an abelian group $A$, is simply the ordinary cohomology of the homotopy quotient
\[ H_G^*(M;A) \= H^*(EG \times_G M; A).\]
A $G$-equivariant map $M \overset{f}\to N$ induces a map on the Borel constructions, giving the desired homomorphisms
\[ H^*_G(N;A) \overset{f_G^*}\longrightarrow H^*_G(M;A).\]

There are important, though subtle, relationships between equivariant and ordinary cohomology.  In particular, $H^*_G(\pt;A) = H^*(BG;A)$, and the natural map $M \to \pt$ makes $H^*_G(M;R)$ into a $H^*(BG;R)$-module when $R$ is a ring.  When $G$ acts freely on $M$, there is a natural isomorphism $H^*_G(M;A) \iso H^*(M/G;A)$.  This follows from the fact that the projection $EG\times_G M \to M/G$ is a locally trivial fiber bundle with fiber $EG$, and hence it is a homotopy equivalence.

\subsection{Principal $G$-bundles with connection}\label{subsec:GBun}Let $G$ be a Lie group, and let $\fg$ be the Lie algebra of left-invariant vector fields on $G$.

A \textit{principal $G$-bundle} on a manifold $X$ is a manifold $P$ equipped with a free right $G$-action and map to $X$ such that
\[ P \overset{\pi}\longrightarrow X\iso P/G.\]
Note that on differential forms, $\pi^*\colon  \O(X) \to \O(P)$ is an injective map, and the subspace $\pi^*\O(X)$ is naturally identified as the \textit{basic} (i.e. invariant horizontal) forms
\begin{equation}\label{eq:HorizontalFormsP} \O(X) \iso  \pi^*\O(X) = \O(X)_{\hor}^G \subset \O(P).\end{equation}
Here, the horizontal forms are defined by: $\omega \in \O(P)_{\hor}$ if  $\iota_X \omega = 0$ for all $X\in \fg$. 

A \textit{connection} on $P$ is an $\Ad$-equivariant $\fg$-valued 1-form $\Theta \in \O^1(P;\fg)$ that restricts fiberwise to the Maurer--Cartan form $\theta_\fg$.  More explicitly,
\begin{itemize}
\item $\iota_X \Theta = X \in \O^0(P;\fg)$ for any vector field $X \in \fg$,
\item $g^*(\Theta) = \Ad_{g^{-1}} \Theta$, where $g\colon  P\to P$ is right multiplication by $g$.
\end{itemize}
The $\Ad$-equivariance can be rewritten as $\Theta \in \big( \O^1(P) \otimes \fg \big)^G$, where the second factor $\fg$ is the adjoint representation of $G$.  The set of all connections on $P$ is an affine space, with $\O^1(X;\fg)$ acting freely and transitively.  The \textit{curvature} is defined 
\[ \Omega \= d\Theta + \tfrac{1}{2} [\Theta \wedge \Theta] \in \O^2(P;\fg).  \]

There is also an equivalent, but more geometric, interpretation of a connection.  The $G$-action defines the distribution $\TV P$ of vertical tangent vectors in $P$ by
\[ \fg \overset{\iso}\longrightarrow \Ker \pi_* =: \TV P  \subset TP.\]
A connection on $P$ is the choice of an equivariant horizontal distribution $\TH P$; i.e. a connection is equivalent to an equivariant splitting of the tangent bundle
\[ TP = \TH P \oplus \TV P \iso \pi^*TM \oplus \fg. \]
The connection 1-form $\Theta$ gives a projection from $TP$ onto $\TV P\iso \fg$, thus defining the horizontal distribution by 
\[ \TH P \= \Ker \Theta \subset TP.\]

The splitting of $TP$ into horizontal and vertical subspaces induces a bi-grading on the differential forms
\[ \O^{i,j}(P) \= C^\infty(P, \L^i \TH P^* \otimes \L^j \TV P^*) \iso C^\infty(P, \pi^* \L^i TX^* \otimes \L^j \fg^* ).\]
In this bi-grading, the exterior derivative decomposes as 
\[ d = d^{0,1} + d^{1,0} + d^{2,-1} \iso  (-1)^i d_{\fg} + d_{\nabla} + (-1)^i \iota_\Omega.\]
Here, $d_{\fg}\colon  \L^j \fg^* \otimes C^\infty(P) \to \L^{j+1}\fg^* \otimes C^\infty(P)$ is the Lie algebra (or Chevalley--Eilenberg) differential for the $G$-module $C^\infty(P)$ \cite{CE48}.  If we restrict to $\O(P)^G$, then $(\L \fg^*, d_\fg)$ is naturally the de Rham complex of left-invariant forms on $G$.  The connection $\Theta$ induces the covariant derivative $d_\nabla$, and $\iota_\Theta$ is the derivation induced by contracting along the vector-valued 2-form $\Omega$.  See Section 3 of \cite{MR3004281} for more details.

\subsection{$G^\star$-algebras}   Let $(\cA, d)$ be a commutative differential graded algebra (DGA), where $d$ is a derivation of degree $+1$ and commutative means $ab = (-1)^{|a||b|}ba$ for homogeneous elements.  A $\operatorname{DGA}$-automorphism of degree 0 is an algebra automorphism $\phi\colon  \cA \to \cA$ that preserves grading and commutes with $d$.

\begin{defn}[Section 2.3 of \cite{GuilleminSternberg99}]\label{defn:G*alg}A $G^\star$-algebra is a commutative DGA $(\cA, d)$ equipped with representations
\[ G \overset{\rho}\longrightarrow \Aut_{\operatorname{DGA}}(\cA) \quad \text{ and }\quad \fg \overset{\iota}\longrightarrow \Der(\cA) \]
of degree 0 and -1, respectively, such that $\iota$ is $G$-equivariant with respect to $\rho$ and satisfies the Cartan equation; i.e. for all $X\in \fg$
\begin{align}
\label{eq:gstar1} \rho_g \iota_X \rho_{g^{-1}} &= \iota_{\Ad_g X}, \\
\label{eq:gstar2} \iota_X d + d\iota_X &= L_X.
\end{align}
Here, $L\colon  \fg \to \Der(\cA)$ is the Lie algebra representation induced by $\rho$.

A map $\phi\colon  \cA_1 \to \cA_2$ is a morphism of $G^\star$-algebras if $\phi$ commutes with multiplication, $\rho$, $d$, and $\iota$.
\end{defn}

\begin{example}\label{ex:G*Forms}Suppose that a manifold $M$ has a left $G$-action.  Then the de Rham complex $\big(\O(M), d \big)$ is naturally a $G^\star$-algebra.  The $G$-action on $\O(M)$ is defined by
\[ \rho_g \omega \= (g^{-1})^*\omega,\]
and $\iota$ is defined by composing the usual interior derivative with the action of $\fg$ on vector fields
\[ \fg \to \mathfrak{X}(M) \overset{\iota}\to \Der( \O(M) ).\]
If $f\colon  M_1 \to M_2$ is $G$-equivariant, then $f^*\colon  \O(M_2) \to \O(M_1)$ is a morphism of $G^\star$-algebras.
\end{example}

%

\begin{rem}\label{rem:LeftVsRight}We use the convention that manifolds $M$ have a left $G$-action, while principal bundles $P$ have a right $G$-action.  However, we sometimes implicitly use the natural switch between left and right actions.  Given a left $G$-action on a set $Y$, define the right $G$-action via the formula $y \cdot g \= g^{-1} \cdot y$; similarly, a right $G$-action induces a left $G$-action.
\end{rem}

%

\begin{defn}Let $\cA$ a $G^\star$-algebra.  An element $a$ is \textit{invariant} if $\rho_g a = a$ for all $g\in G$, and it is \textit{horizontal} if $\iota_X a = 0$ for all $X \in \fg$.  The \textit{basic} sub-algebra is the intersection of the invariant and horizontal elements:
\[ \cA_{\basic} \= \cA^G \cap \cA_{\hor} = \cA_{\hor}^G.\]
\end{defn}

The definition of a $G^\star$-algebra implies that $(\cA_{\basic}, d)$ is a sub-$\DGA$ of $(\cA,d)$.  Furthermore, if $\phi\colon  \cA \to \cB$ is a morphism of $G^\star$-algebras, then $\phi$ restricts to a $\DGA$-morphism on the basic subcomplexes $\phi\colon  \cA_{\bas} \to \cB_{\bas}$.

%
%
%

\begin{defn}A connection on a $G^\star$-algebra $\cA$ is an element $\Theta \in \big(\cA^1 \otimes \fg\big)^G$ such that $\iota_X \Theta = X$ for all $X\in \fg$.  The curvature of $\Theta$ is defined \[ \Omega \= d\Theta + \tfrac{1}{2}[\Theta, \Theta] \in \big( \cA^2 \otimes \fg \big)^G.\]
\end{defn}

\begin{example}Suppose $P\overset{\pi}\to X$ is a principal $G$-bundle.  Example \ref{ex:G*Forms} and Remark \ref{rem:LeftVsRight} show that $\O(P)$ is naturally a $G^\star$-algebra, though we write $\rho_g \omega = g^*\omega$ due to the fact that $G$ acts on the right.  As noted in \eqref{eq:HorizontalFormsP}, the basic subcomplex $\O(P)_{\basic}$ is naturally isomorphic to $\O(X)$.  Furthermore, connections for the principal bundle $P$ are equivalent to connections for the $G^\star$-algebra $\O(P)$.
\end{example}

\begin{example}\label{WeilAlg}
The \textit{Weil algebra} $\W(\fg)$ is a $G^\star$-algebra with connection, and it is constructed so that it canonically maps to any other $G^\star$-algebra with connection.  Explicitly, 
\begin{gather*}\W(\fg) \= S \fg^* \otimes \Lambda \fg^*, \\
\deg S^1\fg^* = 2, \quad \deg \Lambda^1\fg^* = 1, \\
d_{\W} \= d_{\fg} + d_{\K}.
\end{gather*}
Here $S$ and $\Lambda$ are the total symmetric and exterior powers, so the coadjoint representation $\fg^*$ naturally makes $\W(\fg)$ into a $G$-representation.  The differential $d_{\fg}$ is the Chevalley--Eilenberg differential for Lie algebra cohomology with values in the $\fg$-module $S\fg^*$
\[ S^i \fg^*  \otimes \Lambda^j \fg^* \overset{d_{\fg}}\longrightarrow S^i \fg^*  \otimes \Lambda^{j+1} \fg^*;\]
it has degree $(0,1)$ under the bi-grading $\W^{2i,j}(\fg) = S^i \fg^* \otimes \L^j \fg^*$.  The Koszul differential $d_{\K}$ has degree $(2,-1)$ and is defined by extending the natural isomorphism
\[ \Lambda^1 \fg^* \overset{d_{\K}}{\underset{\iso}\longrightarrow} S^1(\fg^*)\]
to a derivation 
\[ S^i(\fg^*) \otimes \Lambda^j(\fg^*) \overset{d_{\K}}\longrightarrow S^{i+1}(\fg^*) \otimes \Lambda^{j-1}(\fg^*).\]
The derivation $\iota$ has degree $(0,-1)$ and is induced by the usual contraction 
\[ \fg \overset{\iota}\longrightarrow \End(\L \fg^*).\]

The Weil algebra $\W(\fg)$ is acyclic; i.e. $H^0(\W(\fg),d_{\W}) = \R$ and $H^i(\W(\fg), d_{\W}) = 0$ for $i>0$.  It has a natural connection
\[ \theta_\fg \in \L^1 \fg^* \otimes \fg = \W^{0,1}(\fg)\otimes \fg \]
given by the identity map $\fg \to \fg$; i.e. $\theta_\fg(X) = X$ for $X\in \fg$.  When there is no risk of confusion, we will drop the subscript and write $\theta$ for $\theta_\fg$.  Using the fact that $d_{\fg} \theta = - \tfrac{1}{2}[\theta \wedge \theta]$, we see that the curvature $\varOmega = \varOmega_{\fg}$ equals $d_{\K} \theta_\fg$.  In light of this, we rewrite the Koszul derivative $d_{\K}$ as $\iota_{\varOmega_\fg}$.
%
\end{example}

\begin{prop}[Weil homomorphism]\label{prop:WeilToConn}Let $\cA$ be a $G^\star$-algebra with connection $\Theta$.  Then, there is natural morphism of $G^\star$-algebras with connection $\W(\fg) \overset{\Theta^*}\longrightarrow \cA$ induced by $\theta \mapsto \Theta$, $\varOmega \mapsto \Omega$.
\end{prop}

Therefore, we see that $\W(\fg)$ serves as a natural algebraic model for differential forms on $EG$.  It is acyclic, and to any $G$-bundle with connection $(P,\Theta) \to X$, the Weil homomorphism is a natural map 
\begin{align}\label{eq:WeilHomo} \begin{split} \W^{2i,j}(\fg) &\overset{\Theta^*}{\longrightarrow} \O^{2i,j}(P) \\
\omega\otimes \eta &\longmapsto \omega( \Omega^{\wedge i}) \wedge \eta(\Theta^{\wedge j}) \end{split}
 \end{align}
that is compatible with the bi-grading, multiplication, $G$-action, derivative $d$, and contraction $\iota$.  Below is a diagram showing this.  For generic $\alpha \in \W(\fg)$, we use the notation $\alpha(\Theta)$ for $\Theta^*(\alpha)$.
{\scriptsize
\[ \hskip -.8in \vcenter{ \xymatrix@C=3mm{ & \vdots& & \vdots& &\vdots \\
&\Lambda^2\fg \ar[u]_{d_\fg} \ar[drr]_<<<<<<<<{\iota_{\varOmega}} & 0
 & S^1 \fg^* \otimes \Lambda^2\fg^* \ar[u]_{d_{\fg}} \ar[drr]_<<<<<<<<{\iota_{\varOmega}}& 0 & S^2 \fg^* \otimes \Lambda^2 \fg^* \ar[u]_{d_{\fg}}   \\
& \fg^*  \ar[u]_{d_\fg} \ar[drr]_<<<<<<<<{\iota_{\varOmega}}^{\iso}& 0 & S^1 \fg^* \otimes \Lambda^1 \fg^* \ar[u]_{d_{\fg}} \ar[drr]_<<<<<<<<{\iota_{\varOmega}}& 0 & S^2 \fg^* \otimes \Lambda^1 \fg^* \ar[u]_{d_{\fg}}  \\
&\R \ar[u]_{0} & 0 & S^1 \fg^* \ar[u]_{d_{\fg}} & 0 & S^2\fg^* \ar[u]_{d_{\fg}} 
}  } \Longrightarrow 
\vcenter{ \xymatrix{ \vdots &\vdots &\vdots &\vdots &\\
\O^{0,2}(P) \ar[u]_{d_\fg} \ar@{.>}[r]^{\nabla} \ar[drr]_<<<<<<<<{\iota_\Omega} & \O^{1,2}(P) \ar@{.>}[u]_{-d_\fg} \ar@{.>}[r]^{d_\nabla} \ar@{.>}[drr]_<<<<<<<<{-\iota_\Omega}& \Omega^{2,2}(P) \ar[u]_{d_\fg} \ar@{.>}[r]^{d_\nabla} \ar[drr]_<<<<<<<<{\iota_\Omega}& \O^{3,2}(P) \ar@{.>}[u]_{-d_\fg} \ar@{.>}[r]^{d_\nabla}&\cdots\\
\O^{0,1}(P)  \ar[u]_{d_\fg} \ar@{.>}[r]^{\nabla} \ar[drr]_<<<<<<<<{\iota_\Omega}& \O^{1,1}(P)\ar@{.>}[u]_{-d_\fg} \ar@{.>}[r]^{d_\nabla} \ar@{.>}[drr]_<<<<<<<<{-\iota_\Omega}& \O^{2,1}(P) \ar[u]_{d_\fg} \ar@{.>}[r]^{d_\nabla} \ar[drr]_<<<<<<<<{\iota_\Omega}& \O^{3,1}(P) \ar@{.>}[u]_{-d_\fg} \ar@{.>}[r]^{d_\nabla}&\cdots\\
\Omega^{0,0}(P) \ar[u]_{d_\fg} \ar@{.>}[r]^{d_\nabla} & \Omega^{1,0}(P) \ar@{.>}[u]_{-d_\fg} \ar@{.>}[r]^{d_\nabla}& \Omega^{2,0}(P) \ar[u]_{d_\fg} \ar@{.>}[r]^{d_\nabla}& \Omega^{3,0}(P) \ar@{.>}[u]_{-d_\fg} \ar@{.>}[r]^{d_\nabla}&\cdots
} } \] }


The universal Chern--Weil and Chern--Simons forms are naturally described via the Weil model.  The usual Chern--Weil forms are given by simply restricting \eqref{eq:WeilHomo} to the basic subcomplex
\begin{align*} (S^k \fg^*)^G = \W^{2k}(\fg)_{\basic} &\xrightarrow{\Theta^*} \O^{2k}(P)_{\basic} \iso \O^{2k}(M) \\
\omega &\longmapsto \omega(\Theta^{\wedge n}).
\end{align*}
That $\W(\fg)$ is acyclic implies any such $\omega \in (S\fg^*)^G$ is exact in $\W(\fg)$, and there is a standard way to pick out such a coboundary.  Using the trivial $G^\star$-algebra $\O([0,1])$ with natural coordinate $t$, define 
\begin{align*}
\theta_t &\= t\theta & &\in \big( \O([0,1]) \otimes \W(\fg)  \big)^1\otimes \fg, \\
\varOmega_t &\= d(\theta_t) + \tfrac{1}{2} [ \theta_t \wedge \theta_t]
& &\in \big( \O([0,1]) \otimes \W(\fg) \big)^2 \otimes \fg, 
\end{align*}
where one can rewrite the curvature to give $\varOmega_t =  dt\,\theta + t \varOmega + \tfrac{1}{2}(t^2-t)[\theta \wedge \theta]$.

For $\omega \in (S^k \fg^*)^K$, then $\omega(\varOmega_t^{\wedge k}) \in \big( \O([0,1]) \otimes \W(\fg)  \big)_{\basic}$, and the Chern--Simons form is defined using integration over the interval $[0,1]$ by 
\begin{equation}\label{eq:ChernSimonsWeil}
\CS_\omega \= \int_{[0,1]} \omega \big( \varOmega_t^{\wedge k} \big) \in \W^{2k-1}(\fg)^G.
\end{equation}
Stokes Theorem implies that $d\CS_\omega = \omega \in \W(\fg)$, and the Weil homomorphism sends this to the standard Chern--Simons forms on the total space of principal bundles
\begin{align*} \W(\fg)^G &\xrightarrow{\Theta^*} \O(P)^G \\
\CS_\omega &\longmapsto \CS_\omega(\Theta).
\end{align*}

\subsection{Equivariant Forms}\label{subsec:deRhamModel}
We now describe an explicit de Rham model $\O_G(M)$, often referred to as the Weil model, for $H^*_G(M;\R)$.  While the definition of $\O_G(M)$ does not require $G$ to be compact, the cohomology of $\O_G(M)$ is not necessarily isomorphic to $H^*_G(M;\R)$ for non-compact $G$.

Suppose $M$ is a manifold equipped with a smooth left $G$-action.  The tensor product $\W(\fg) \otimes \O(M)$ is a $G^\star$-algebra with derivative $d_G = d_{\W}\otimes 1 + 1 \otimes d$.  

\begin{defn}The complex of \textit{equivariant forms} is the basic sub-complex
\begin{equation*} \left(  \O_G^*(M), d_G \right) \= \big( \left( \W(\fg)\otimes \O(M) \right)^G_{\hor}, \> d_{\W} \otimes 1 + 1 \otimes d \big). \end{equation*}
\end{defn}

Note that if $f\colon  M \to N$ is $G$-equivariant map, then 
\[\W(\fg)\otimes \O(N) \xrightarrow{1 \otimes f^*} \W(\fg) \otimes \O(M)\] restricts to the basic complex, giving a naturally induced DGA-morphism we denote
\[ \O_G(N) \overset{f^*_G}\longrightarrow \O_G(M).\]

To relate $\O_G(M)$ to the cohomology of $EG \times_G M$, we use the following geometric fact.  There exist finite-dimensional smooth $n$-classifying bundles with connection $(E^{(n)}G, \Theta_{E^{(n)}G}) \to B^{(n)}G$ (\cite{MR0133772}).  This means that if $\dim(X) \leq n$, any bundle with connection $(P,\Theta)\to X$ is isomorphic to the pullback $f^*(E^{(n)}G, \Theta_{E^{(n)}G})$ for some smooth map $X \overset{f}\to B^{(n)}G$.  Furthermore, any two maps classifying $(P,\Theta)$ are homotopic.  The universal bundle $EG\to BG$ can then be constructed as a direct limit of finite-dimensional manifolds.
\[ \vcenter{ \xymatrix{  (EG,\Theta_{EG}) \ar[d]^{\pi} \\ BG } }  \quad \= \quad \varinjlim
\left(\vcenter{\xymatrix{ (E^{(n)}G, \Theta_{E^{(n)}G}) \ar[d]^{\pi} \\ B^{(n)}G } } \right) \]
When discussing differential forms and cochains, we use the notations
\begin{align*} \O^*(EG) &= \varprojlim \O^*(E^{(n)}G), & C^*(EG;A) &= \varprojlim C^*(E^{(n)}G;A),\\\O^*(BG) &= \varprojlim \O^*(B^{(n)}G), & C^*(BG;A) &= \varprojlim C^*(B^{(n)}G;A).\end{align*}

The connection $\Theta_{EG}$ and the Weil homomorphism give a $G^\star$-algebra homomorphism 
\begin{align*}
\W(\fg) \xrightarrow{\Theta_{EG}^*} \O(EG),
\end{align*}
which in turn gives
\begin{equation}\label{eq:EquivToBorel} \O_G(M) = \big(\W(\fg) \otimes \O(M) \big)_{\basic} \xrightarrow{\Theta_{EG}^* \otimes 1}\O(EG \times M)_{\basic} \iso \O(EG\times_G M).  \end{equation}

\begin{thm}[Equivariant de Rham Theorem]For $G$ compact, \eqref{eq:EquivToBorel} induces an isomorphism in cohomology
\[ H^*\big( \O_G(M), d_G \big) \overset{\iso}\longrightarrow H^*\big( \O(EG\times_G M),d \big) \iso H^*_G(M;\R).\]
\end{thm}

A proof of the above theorem can be found in Theorems 2.5.1 (and Theorems 4.3.1 and 6.7.1) of \cite{GuilleminSternberg99}.  The following lemma is used in the first construction of $\HG^*(M)$ in Section \ref{sec:HS}, along with the proof of \ref{prop:Uniqueness}.  Though it is certainly well-known, we are unaware of a specific reference, and we prove it directly so that the first construction of $\HG^*(M)$ does not rely on results from \cite{MR3049871}.

\begin{lemma}\label{lem:WeilEGinj}The homomorphism $\W(\fg) \xrightarrow{\Theta_{EG}^*} \O(EG)$ is injective, as are the induced homomorphisms $\O_G(M) \to \O(EG\times_G M)$.
\end{lemma}
\begin{proof}The second homomorphism is given by restricting 
\[ \W(\fg) \otimes \O(M) \longrightarrow \O(EG)\otimes \O(M) \iso \O(EG\times M)\]
to the basic subcomplex.  Showing this map is injective is equivalent to showing $\W(\fg) \to \O(EG)$ is injective.

Any bundle with connection is isomorphic to the pullback of $(EG,\Theta_{EG})$, so it suffices to show that for any element $\alpha \in \W(\fg)$, there exists some $(P,\Theta) \xrightarrow{\pi} X$ such that $\alpha(\Theta) \neq 0 \in \O(P)$.  At any point $p$, $\Lambda T_p^*P$ is an algebra freely generated by the horizontal and vertical cotangent spaces.  On the vertical part, $\W^{0,*}(\fg) \to \L T_p^V P^* \iso \L \fg^*$  is an isomorphism, so it suffices to show the Weil homomorphism is injective on $\W^{*,0}(\fg)$. 

Let $\omega \in S^n \fg^*$ be any non-zero element.  Then $\omega( \xi_{i_1}\cdots \xi_{i_n}) \neq 0$ for some $(i_1, \ldots, i_n)$, where $\{\xi_i\}$ be a basis of $\fg$.  Let $\R^{2n} \times G \to \R^{2n}$ be the trivial bundle.  In a neighborhood of $0$, use the canonical frame $p$ to define a connection $\Theta$ by 
\[ p^*\Theta = x^1 dx^2 \xi_{i_1} + \cdots + x^{2n-1} dx^{2n} \xi_{i_n} \in \O^1 (\R^{2n};\fg).\]
The local curvature is given by 
\begin{align*} p^*\Omega &= dx^1 dx^2 \xi_{i_1} + \cdots + dx^{2n-1} dx^{2n} \xi_{i_n} +  \tfrac{1}{2}[p^*\Theta \wedge p^*\Theta],
\end{align*}
and the terms involving $p^*\Theta$ vanish when all $x^i=0$.  Evaluating $p^*\omega(\Theta^{\wedge n})$ at the origin gives 
\begin{align*}
p^*\omega(\Theta^{\wedge n})_{(0)}\big( \partial_1, \ldots, \partial_{2n}  \big) = n! \omega( \xi_{i_1} \cdots \xi_{i_n} ) \neq 0. 
\end{align*} 
\qedhere
\end{proof}

\subsection{Cartan model and Matthai--Quillen}\label{subsec:CartanMQ}There is another $G^\star$-algebra structure on $\W(\fg)\otimes \O(M)$ that leads to the Cartan model for $H^*_G(M;\R)$.  It is obtained by modifying the $G^\star$-structure so that the interior derivative $\iota$ only acts on the factor $\W(\fg)$.

The $G$-equivariant homomorphism
\[ \O^1(M) \overset{\iota_{\theta}}\longrightarrow \L^1 \fg^*, \]
defined by $(\iota_\theta \psi) (X) = \iota_{\theta(X)} \psi = \iota_X \psi$,  induces a derivation on associative algebra $\W(\fg)\otimes \O(M)$.  It exponentiates to a $G$-equivariant automorphism of associative algebras
\[ \W(\fg)\otimes \O(M) \xrightarrow{\exp(\iota_\theta)} \W(\fg)\otimes \O(M) \]
known as the Mathai--Quillen isomorphism \cite{MathaiQuillen}.  Conjugating by $e^{\iota_\theta}$ gives the new operators 
\[ d_C \= (\exp{\iota_{\theta}})\, d_G\, (\exp{- \iota_\theta}), \quad \iota^C \= (\exp{\iota_{\theta}} )\, \iota \, (\exp - \iota_\theta) \]
where $\iota^C_X (\alpha \otimes \psi) = (\iota_X \alpha) \otimes \psi$.  The isomorphism between $G^\star$-algebras induces an isomorphism of the basic subcomplexes, one of which is the previously discussed Weil model.  The other, known as the Cartan model, is 
\[ \big( \W(\fg) \otimes \O(M)\big)^G_{\Ker \iota^C} =  \big( S\fg^* \otimes \O(M) \big)^G,\]
and the derivative takes the form  $d_C = d - \iota_{\varOmega_\fg}$,
where $\iota_{\varOmega_\fg}$ is the degree 1 derivation induced by 
\[ \O^k(M) \xrightarrow{\varOmega_\fg \otimes 1} S^1 \fg^* \otimes \fg \otimes \O^k(M) \xrightarrow{1 \otimes\iota} S^1 \fg^* \otimes \O^{k-1}(M).\]
The Mathai--Quillen isomorphism may be used to interpret $\O_G(M)$ as the Cartan model throughout the entire paper.

\bibliographystyle{alphanum}
\bibliography{MyBibDesk}

\begin{thebibliography}{BNV}

\bibitem[BNV]{1311.3188}
Ulrich Bunke, Thomas Nikolaus, and Michael V{\"o}lkl.
\newblock Differential cohomology theories as sheaves of spectra.
\newblock {\em Journal of Homotopy and Related Structures}, pages 1--66, 2014.

\bibitem[Bry]{Bry93}
Jean-Luc Brylinski.
\newblock {\em Loop spaces, characteristic classes and geometric quantization},
  volume 107 of {\em Progress in Mathematics}.
\newblock Birkh\"auser Boston Inc., Boston, MA, 1993.

\bibitem[BT]{MR1850463}
Raoul Bott and Loring~W. Tu.
\newblock Equivariant characteristic classes in the {C}artan model.
\newblock In {\em Geometry, analysis and applications ({V}aranasi, 2000)},
  pages 3--20. World Sci. Publ., River Edge, NJ, 2001.

\bibitem[Bun]{1208.3961}
Ulrich Bunke.
\newblock {D}ifferential cohomology, 2012.
\newblock arXiv:1208.3961.

\bibitem[BV]{MR705039}
Nicole Berline and Mich{\`e}le Vergne.
\newblock Z\'eros d'un champ de vecteurs et classes caract\'eristiques
  \'equivariantes.
\newblock {\em Duke Math. J.}, 50(2):539--549, 1983.

\bibitem[CE]{CE48}
Claude Chevalley and Samuel Eilenberg.
\newblock Cohomology theory of {L}ie groups and {L}ie algebras.
\newblock {\em Trans. Amer. Math. Soc.}, 63:85--124, 1948.

\bibitem[CS]{CS85}
Jeff Cheeger and James Simons.
\newblock Differential characters and geometric invariants.
\newblock In {\em Geometry and topology (College Park, Md., 1983/84)}, volume
  1167 of {\em Lecture Notes in Math.}, pages 50--80. Springer, Berlin, 1985.

\bibitem[FH]{MR3049871}
Daniel~S. Freed and Michael~J. Hopkins.
\newblock Chern--{W}eil forms and abstract homotopy theory.
\newblock {\em Bull. Amer. Math. Soc. (N.S.)}, 50(3):431--468, 2013.

\bibitem[FOS]{MR1311654}
Jos{\'e}~M. Figueroa-O'Farrill and Sonia Stanciu.
\newblock Gauged {W}ess-{Z}umino terms and equivariant cohomology.
\newblock {\em Phys. Lett. B}, 341(2):153--159, 1994.

\bibitem[Fre]{Fre95}
Daniel~S. Freed.
\newblock Classical {C}hern-{S}imons theory. {I}.
\newblock {\em Adv. Math.}, 113(2):237--303, 1995.

\bibitem[FSS]{MR3019405}
Domenico Fiorenza, Urs Schreiber, and Jim Stasheff.
\newblock \v {C}ech cocycles for differential characteristic classes: an
  {$\infty$}-{L}ie theoretic construction.
\newblock {\em Adv. Theor. Math. Phys.}, 16(1):149--250, 2012.

\bibitem[Gom]{MR2147734}
Kiyonori Gomi.
\newblock Equivariant smooth {D}eligne cohomology.
\newblock {\em Osaka J. Math.}, 42(2):309--337, 2005.

\bibitem[GS]{GuilleminSternberg99}
Victor~W. Guillemin and Shlomo Sternberg.
\newblock {\em Supersymmetry and equivariant de {R}ham theory}.
\newblock Mathematics Past and Present. Springer-Verlag, Berlin, 1999.
\newblock With an appendix containing two reprints by Henri Cartan [ MR0042426
  (13,107e); MR0042427 (13,107f)].

\bibitem[Hei]{MR2206877}
J.~Heinloth.
\newblock Notes on differentiable stacks.
\newblock In {\em Mathematisches {I}nstitut, {G}eorg-{A}ugust-{U}niversit\"at
  {G}\"ottingen: {S}eminars {W}inter {T}erm 2004/2005}, pages 1--32.
  Universit\"atsdrucke G\"ottingen, G\"ottingen, 2005.

\bibitem[HL]{HL06}
Reese Harvey and Blaine Lawson.
\newblock From sparks to grundles---differential characters.
\newblock {\em Comm. Anal. Geom.}, 14(1):25--58, 2006.

\bibitem[HQ]{MR3335251}
Michael~J. Hopkins and Gereon Quick.
\newblock Hodge filtered complex bordism.
\newblock {\em J. Topol.}, 8(1):147--183, 2015.

\bibitem[HS]{HS05}
M.~J. Hopkins and I.~M. Singer.
\newblock Quadratic functions in geometry, topology, and {M}-theory.
\newblock {\em J. Differential Geom.}, 70(3):329--452, 2005.

\bibitem[KN]{KN63}
Shoshichi Kobayashi and Katsumi Nomizu.
\newblock {\em Foundations of differential geometry. {V}ol {I}}.
\newblock Interscience Publishers, a division of John Wiley \& Sons, New
  York-Lond on, 1963.

\bibitem[K{\"u}b]{1510.06392v1}
Andreas K{\"u}bel.
\newblock {E}quivariant {D}ifferential {C}ohomology, 2015.
\newblock arXiv:1510.06392.

\bibitem[LMS]{MR711050}
R.~K. Lashof, J.~P. May, and G.~B. Segal.
\newblock Equivariant bundles with abelian structural group.
\newblock In {\em Proceedings of the {N}orthwestern {H}omotopy {T}heory
  {C}onference ({E}vanston, {I}ll., 1982)}, volume~19 of {\em Contemp. Math.},
  pages 167--176, Providence, RI, 1983. Amer. Math. Soc.

\bibitem[Lur]{MR2522659}
Jacob Lurie.
\newblock {\em Higher topos theory}, volume 170 of {\em Annals of Mathematics
  Studies}.
\newblock Princeton University Press, Princeton, NJ, 2009.

\bibitem[Mei]{Meinrenken-EquivCoh}
E.~Meinrenken.
\newblock Equivariant cohomology and the {C}artan model.
\newblock In {\em Encyclopedia of Mathematical Physics}. Elsevier, 2006.

\bibitem[MQ]{MathaiQuillen}
Varghese Mathai and Daniel Quillen.
\newblock Superconnections, {T}hom classes, and equivariant differential forms.
\newblock {\em Topology}, 25(1):85--110, 1986.

\bibitem[NR]{MR0133772}
M.~S. Narasimhan and S.~Ramanan.
\newblock Existence of universal connections.
\newblock {\em Amer. J. Math.}, 83:563--572, 1961.

\bibitem[Red]{MR3004281}
Corbett Redden.
\newblock Harmonic forms on principal bundles.
\newblock {\em Asian J. Math.}, 16(4):637--660, 2012.

\bibitem[Sch]{1310.7930}
Urs Schreiber.
\newblock {D}ifferential cohomology in a cohesive infinity-topos, 2013.
\newblock arXiv:1310.7930.

\bibitem[SS]{SS08a}
James Simons and Dennis Sullivan.
\newblock Axiomatic characterization of ordinary differential cohomology.
\newblock {\em J. Topol.}, 1(1):45--56, 2008.

\end{thebibliography}

\end{document}